\documentclass[9pt]{amsart}
\usepackage{amsmath, amsthm, amscd, amsfonts,graphicx}
\usepackage{amsmath,amsthm,amsthm, amscd, amsfonts, stackrel, latexsym, dsfont,amssymb,mathrsfs,textcomp,wasysym,hyperref}
\usepackage{tikz}
\usepackage{tikz-cd}
\usetikzlibrary{matrix,arrows}
\usepackage[all]{xy}

\usepackage{color} 
\definecolor{slateblue}{rgb}{0,0.0,0.8}

\usepackage{xcolor}
\usepackage{graphicx}
\hypersetup{%
	colorlinks=true,
	linkcolor=slateblue,
	linkbordercolor=gray,
	citecolor=slateblue
}

\theoremstyle{definition}
\newtheorem{definition}{Definition}[section]
\newtheorem{example}[definition]{Example}
\theoremstyle{remark}
\newtheorem{remark}[definition]{Remark}
 
\theoremstyle{plain}
\newtheorem{theorem}[definition]{Theorem}

\newtheorem{proposition}[definition]{Proposition}
\newtheorem{corollary}[definition]{Corollary}
\newtheorem{notation}[definition]{Notation}
\newtheorem{construction}[definition]{Construction}
\newtheorem{setup}[definition]{Setup}
\usepackage{mathptmx}

\makeatletter
\@namedef{subjclassname@2020}{\textup{2020} Mathematics Subject Classification}
\makeatother

\newtheorem{Def}{Definition}[section]

\usepackage{amsthm}

\newtheorem{innercustomgeneric}{\customgenericname}
\providecommand{\customgenericname}{}
\newcommand{\newcustomtheorem}[2]{%
	\newenvironment{#1}[1]
	{%
		\renewcommand\customgenericname{#2}%
		\renewcommand\theinnercustomgeneric{##1}%
		\innercustomgeneric
	}
	{\endinnercustomgeneric}
}

\newcustomtheorem{customthm}{Theorem}
\newcustomtheorem{customlemma}{Lemma}

\begin{document}
	
	
	\title[Diffeological \v{C}ech cohomology]{Diffeological \v{C}ech cohomology}

	\author[A. Ahmadi]{Alireza Ahmadi}
	\address{Alireza Ahmadi, Department of Mathematical Sciences, Yazd University, 89195--741, Yazd, Iran }
	\email{ahmadi@stu.yazd.ac.ir, alirezaahmadi13@yahoo.com}




\maketitle
\begin{abstract}
	Motivated by problems in which data are given over covering generating families,
	we suggest a new cohomology theory for diffeological spaces, called diffeological  \v{C}ech cohomology,
	which is an exact $ \partial $-functor of the section functor for sheaves on diffeological spaces. As applications,  under the situations of a setup, i) the generalized Mayer-Vietoris sequence for a diffeological space is established; ii) a version of the de Rham theorem is obtained, which connects diffeological \v{C}ech cohomology to the de Rham cohomology. Moreover, we characterize the isomorphism classes of diffeological fiber, principal, and vector bundles as (non-abelian) diffeological \v{C}ech cohomology in degree 1.
\end{abstract}
\tableofcontents
\section{Introduction}
\v{C}ech cohomology provides a  method to get global information from local data by extending the
section functor of sheaves.
One elegant sample of this abstract approach occurs in the differential geometry of manifolds to
investigate the differential equation $ d\alpha=\omega $, for a given closed differential form $ \omega $.
As manifolds are locally contractible, one can locally solve the equation to obtain local data. In contrast, the de Rham cohomology is a natural approach to this problem, which plays the role of an invariant for manifolds and is connected to \v{C}ech cohomology under the de Rham theorem.

The purpose of this paper is to study such a problem in diffeology --- a convenient framework to do differential geometry on mapping spaces, infinite dimensional spaces, and singular spaces like orbifolds, irrational tori,  beyond the classical setting of manifolds \cite{BH}. A diffeological space is given by a set $ X $ along with a structure of \textit{diffeology}  whose elements are called \textit{plots}.
In fact, a diffeology determines which parametrizations (i.e., maps from open subsets of Euclidean
spaces) are plots by the axioms of \textit{covering}, \textit{smooth compatibility}, and
\textit{locality} that are reminiscent of properties of ordinary smooth maps.
This theory was introduced by J.-M. Souriau \cite{JMS} in the 1980s, and the main reference is the textbook  \cite{PIZ2013} by P. Iglesias-Zemmour.

Differential forms and the de Rham cohomology have nicely been developed to diffeological spaces.
However, one might not be able to solve locally the differential equation $ d\alpha=\omega $ on a diffeological space $ X $ more complicated than manifolds. Instead, one can always solve the pullbacks of the equation along global plots $ \mathbb{R}^n\rightarrow X $, $n$ ranges over
nonnegative integers, which constitute a covering generating family for $ X $. Of course, one can take any other covering generating family whose elements are defined on contractible domains\footnote{Compare with the restrictions of the differential equation on a contractible open cover in the classical setting of manifolds.}.
In other words,  solutions (data) are obtained over covering generating families, which may not be compatible to construct a global solution.

To explore the given data through the \v{C}ech machinery, we introduce a modified version of \v{C}ech cohomology for diffeological spaces, called diffeological \v{C}ech cohomology with coefficients in a (pre)sheaf. Diffeological \v{C}ech cohomology is formulated in terms of covering generating families to detect compatible data on them in the zero degree (Proposition \ref{p-0th}), and gives rise to an exact $ \partial $-functor of  the section functor for sheaves on diffeological spaces (Proposition \ref{p-long}).
It is also dual to quasi-\v{C}ech homology defined in \cite{AD}.

To investigate the relationship between diffeological \v{C}ech cohomology and the de Rham cohomology, we consider a setup (Setup \ref{setting-1}) in which the generalized Mayer-Vietoris sequence for diffeological spaces is achieved (Proposition \ref{prop-mv}). 
This leads to a version of the de Rham theorem in diffeology:
\begin{customthm}{I}
	Let $ X $ be a diffeological space. In the situation of Setup \ref{setting-1},
	if for every covering generating family $\mathcal{C}$ of $ X $, there exists a partition of unity subordinate to $\mathcal{C}$,
	then
	$ \check H^n(X;\mathbb{R})\cong H_{dR}^n(X) $, where $ \mathbb{R} $
	denotes the locally constant sheaf on $ X $.
\end{customthm}
In particular, diffeological \'{e}tale manifolds fulfill the situation of this setup, and
for a usual manifold, this version of the de Rham theorem is true.

Another aspect of diffeological \v{C}ech cohomology concerns the classifying of diffeological fiber bundles.
We first consider diffeological fiber bundles with gauge groups by the language of trivialization and transition maps, which include the usual diffeological fiber, principal, and vector bundles as special cases.
The transition maps of a diffeological fiber bundle satisfy the diffeological \v{C}ech 1-cocycle condition. Analogous to the classical situation, this leads to a bijection between 
the first diffeological \v Cech cohomology classes and isomorphism classes of diffeological fiber bundles:
\begin{customthm}{II}
	Assume that $ X $ and $ T $ are diffeological spaces and let 
	$ \varrho:G\rightarrow \mathrm{Diff}(T) $ be a smooth faithful action of a (not necessarily abelian) diffeological group $ G $ on $ T $.
	There is a bijective correspondence
	\begin{center}
		$ \Theta:\check{H}^1(X,G)\rightarrow\mathrm{Bund}_T(X,G) $
	\end{center}
	between the first diffeological \v Cech cohomology classes $ \check{H}^1(X,G) $ with the coefficients in the sheaf of smooth functions in $ G $ on $ X $ and the set of isomorphism classes of diffeological fiber bundles with gauge group $ G $ 
	and typical fiber $ T $ on $ X $.
\end{customthm}

\paragraph{\textbf{Other related works.}}
There are other kinds of \v{C}ech cohomology for diffeological spaces in the literature.
In \cite{PIZ2022}, a version of \v{C}ech cohomology is defined by P. Iglesias-Zemmour, as the
Hochschild cohomology of some gauge monoid, associated with the specific covering generating family of round plots, with coefficients in $ \mathbb{R} $. In this context, a \v{C}ech-de Rham double complex is constructed and discussed.

In \cite{DA2018}, three versions of \v{C}ech cohomologies are suggested as those on the site of plots, the site of D-open subsets, and as the cohomology of the associated cochain complex  of
sections of a presheaf. In Appendix \ref{A}, we show that the last two are the same. 

Furthermore, in \cite{KWW} another \v{C}ech cohomology is given by Krepski et al. in terms of the nebula of covering generating families, which detects the basic sections on the nebula in degree $ 0 $ (see Appendix \ref{A}). It is also proved that the first degree \v{C}ech cohomology classes for the sheaf of smooth functions to an abelian diffeological group
$ G $ classify diffeological principal $ G $-bundles.
\paragraph{\textbf{Organization.}} The paper is organized as follows. 
Section \ref{S2} is devoted to an overview of diffeological spaces and diffeological fiber bundles. In Section \ref{S3}, we recall presheaves and sheaves on diffeological spaces. In particular, a framework based on presheaves is provided
to exhibit  cohomology theories.
In Section \ref{S4}, diffeological \v{C}ech cohomology is introduced and its behavior and properties are discussed.
In Section \ref{S5}, the generalized Mayer-Vietoris sequence for  diffeological spaces is  established and a version of the de Rham theorem is given.
In particular, we show that
this version of the de Rham theorem holds for a usual manifold.
Moreover, we present the diffeological \v{C}ech-de Rham double complex and its spectral sequences.
In Section \ref{S6}, diffeological fiber bundles with gauge groups are defined.
We characterize the isomorphism classes of diffeological fiber, principal, and vector bundles as (non-abelian) diffeological \v{C}ech cohomology in degree 1.
In Appendix \ref{A}, we conclude the paper  with a detailed discussion on the \v{C}ech cohomologies suggested in \cite{DA2018}, as well as,
a remark on the \v{C}ech cohomology due to Krepski et al.

\section{Background on diffeological spaces}\label{S2}
We briefly recall the required definitions and results on diffeological spaces from \cite{PIZ2013}.
%
\subsection{Basic definitions}
\begin{definition}
	An $n$-\textbf{domain}, for a nonnegative integer $n$, is an open subset of  Euclidean space $\mathbb{R}^n$ with the standard topology. 
	Any map from a domain to a set $X$ is said to be a \textbf{parametrization} in $X$.  
	The only $0$-parametrization with the value $x\in X$ is denoted by the bold letter $ \mathbf{x}$.
	A parametrization $ P $ in $ X $ with $ 0\in\mathrm{dom}(P) $ and $ P(0)=x $ is called a parametrization \textbf{centered} at $ x $.
	If the domain of definition of a parametrization $ P $, denoted by 
	$ \mathrm{dom}(P) $, is an $n$-domain, then $ P $ is called an $n$-parametrization. 
	A family $\lbrace P_i: U_i\rightarrow X\rbrace_{i\in J}$ of $n$-parametrizations is \textbf{compatible} if
	$P_i|_{U_i\cap U_j}=P_j|_{U_i\cap U_j}$, for all $i, j\in J$. 
	For such a family, the parametrization 
	$P:\bigcup_{i\in J} U_i\rightarrow X$ given by $P(r)=P_i(r)$ for $r\in U_i$, 
	is said to be the \textbf{supremum} of the family. 
	By convention, the supremum of the empty family is the empty parametrization $ \varnothing\rightarrow X $.
\end{definition} 

\begin{definition} 
	A \textbf{diffeology} $ \mathcal{D} $ on a set $ X $ is a collection of parametrizations in $ X $ with the following axioms:
	\begin{enumerate}
		\item[\textbf{D1.}] 
		The union of the images of the elements of $ \mathcal{D} $ covers $ X $.
		
		\item[\textbf{D2.}] 
		For every element $P:U\rightarrow X$ of $\mathcal{D}$ and every smooth map $F:V\rightarrow U$ between domains, the parametrization $P\circ F$ belongs to $\mathcal{D}$.
		\item[\textbf{D3.}]
		The supremum of any compatible family of elements of $\mathcal{D}$ also belongs to $\mathcal{D}$.
	\end{enumerate}
	A \textbf{diffeological space} $ (X,\mathcal{D}) $ is an underlying set $X$ equipped with a
	diffeology $\mathcal{D}$, whose elements are called the \textbf{plots} in $ X $.
	A diffeological space is just denoted by the underlying set, when the diffeology is understood.
\end{definition} 

\begin{definition} 
	A \textbf{prediffeology} on a set $ X $ is a set $ \mathcal{P} $ of parameterizations in $ X $ satisfying D1 and D2.
	A \textbf{parametrized cover} of $ X $ is a set $ \mathcal{C} $ of parameterizations in $ X $ satisfying D1.
\end{definition} 

\begin{definition}
	Let  $X$ be a set and let $\mathcal{C}$ be a parametrized cover of $X$.
	The \textbf{prediffeology generated} by $\mathcal{C}$ denoted by $\lfloor\mathcal{C}\rfloor$, consists of parametrizations $ P\circ F $, where $ P $ is an element of $\mathcal{C}$ and $ F $ is a smooth map between domains.
	The \textbf{diffeology generated} by $ \mathcal{C} $, denoted by $\langle\mathcal{C}\rangle$, is the set of  parametrizations $ P $ 
	that are the supremum of a compatible family
	$ \lbrace P_i\rbrace_{i\in J} $ of parametrizations in $ X $ with $ P_i\in\lfloor\mathcal{C}\rfloor $.
	For a diffeological space $ (X,\mathcal{D}) $,
	a \textbf{covering generating family} is a parametrized cover $\mathcal{C}$ of $ X $  generating the diffeology of the space, i.e., $\langle\mathcal{C}\rangle=\mathcal{D}$.
	Denote by $ \mathsf{CGF}(X) $ the collection of all covering generating families of the space $ X $.
\end{definition}

\begin{example}
	For any diffeological space $ X $, each of the following collections is a covering generating family:
	\begin{enumerate}
		\item[$ {\tiny\blacktriangleright}  $] 
		The diffeology of the space,
		\item[$ {\tiny\blacktriangleright}  $] 
		The collection of plots whose domains are open balls,
		\item[$ {\tiny\blacktriangleright}  $] 
		The collection of global plots $ \mathbb{R}^n\rightarrow X $ ($n$ ranges over
		nonnegative integers),
		\item[$ {\tiny\blacktriangleright}  $] 
		The collection of centered plots, i.e., plots $ U\rightarrow X $ with $ 0\in U $.
	\end{enumerate} 
\end{example} 

\begin{definition}
	Let $X$ and $Y$ be two diffeological spaces. A map $f:X\rightarrow Y$ is \textbf{smooth} if for every plot $P$ in $X$, the composition $f\circ P$ is a plot in the space $Y$.
	Denote by $ \mathsf{Diff} $ the category of diffeological spaces and smooth maps.
	The isomorphisms in the category $ \mathsf{Diff} $ are called \textbf{diffeomorphisms}.
\end{definition} 

Plots in a diffeological space are exactly the smooth parametrizations in the space.

\begin{definition}
	A map $f:X\rightarrow Y$ between diffeological spaces is a \textbf{subduction} if 
	the set
	$ \lbrace f\circ P\mid P $ is a plot in $ X\rbrace $
	is a covering generating family for $ Y $. Obviously, every subduction is a surjective smooth map.
\end{definition}
\begin{proposition}\label{pro-subd-smd}
	(\cite[\S 1.51]{PIZ2013}).
	Let $\pi:X'\rightarrow X$ be a subduction. A map $ f:X\rightarrow Y $ is smooth if and only
	if $ f\circ\pi $ is smooth. Moreover, the map $ f $ is a subduction if and only if $ f\circ\pi $ is a subduction.
\end{proposition}

\begin{definition}
	Let $ X $ be a diffeological space.
	If $ R $ is an equivalence relation on $ X $ and $ q:X\rightarrow X/R $ is the quotient map, the diffeology generated by the family 
	$ \lbrace q\circ P\mid P $ is a plot in $ X\rbrace $ 
	is called the \textbf{quotient diffeology} on $ X/R $ and $ X/R $ is the quotient space.
	In this situation, the quotient map $ q:X\rightarrow X/R $ is a subduction.
\end{definition}

\begin{definition}
	Let $ X $ be a diffeological space. A \textbf{diffeological subspace} of $ X $ is a subset $ X'\subseteq X $ equipped with the \textbf{subspace diffeology}, which is the set of all plots in $ X $ with values in $ X' $.
\end{definition}
\begin{definition}
	The \textbf{product diffeology} on
	$ X=\prod_{i\in J} X_i$ is given by the parametrizations $ P $ in $ X $ for which
	$ \pi_i\circ P$ is a plot in $X_i$ for all $ i\in J $, where
	$ \pi_i:X\rightarrow X_i$ is the natural projection.
	If $ J=\lbrace1,\ldots,n\rbrace $, then the plots in the product $X=X_1\times\dots\times X_n $ are
	$ n $-tuples $ (P_1,\ldots,P_n) $ where each $ P_i $ is a plot in $ X_i $.
\end{definition}

\begin{definition}
	Let $ \lbrace X_i\rbrace_{i\in J} $ be a family of diffeological spaces. 
	The \textbf{sum diffeology} on the direct sum $ X=\bigsqcup_{i\in J} X_i $ is given by following
	property: A parametrization $ P:U\rightarrow X $ is a plot if there exists a partition $ \lbrace U_i\rbrace_{i\in J} $ of $ U $ such that $ P_i=P|_{U_i} $ is a plot in $ X_i $.
\end{definition}

\begin{definition}
	Let $ X $ be a diffeological space and let $ \mathcal{C} $ be a parametrized cover of $ X $.
	The sum space $ \mathsf{Nebula}(\mathcal{C})=\bigsqcup_{P\in \mathcal{C}} \mathrm{dom}(P) $ is called the \textbf{nebula} of the parametrized cover $ \mathcal{C} $.
\end{definition}

\begin{definition}
	Every diffeological space $ X $ has a natural topology called the D-\textbf{topology} in which a subset of $X$ is D-\textbf{open} if its preimage by any plot is open. 
	We denote a diffeological space $ X $ endowed with the D-topology by $ D(X) $.
\end{definition}

Any smooth map is D-continuous, that is, continuous with respect to the D-topology \cite[\S 2.9]{PIZ2013}

\subsection{Algebraic diffeological objects}
\begin{definition}
	A \textbf{diffeological monoid} is a monoid along with a \textbf{monoid diffeology} such that the multiplication is
	smooth.
\end{definition}

\begin{definition}
	A \textbf{diffeological group} is a group equipped with a \textbf{group diffeology} such that the multiplication and the inversion are
	smooth.
\end{definition}

\begin{example}
	Let $ X $ be a diffeological space. Let $ \mathrm{Diff}(X) $ denote the group of diffeomorphisms on $ X $ with the composition operation.
	A parametrization $ P : U \rightarrow\mathrm{Diff}(X) $ is a plot for the \textbf{standard
		diffeology} of group of diffeomorphisms if $ P $ and $ r\mapsto P(r)^{-1} $ are plots for functional
	diffeology (see \cite[1.61]{PIZ2013}). Then  $ \mathrm{Diff}(X) $  is a diffeological group thanks to the smoothness of the composition.
\end{example}

\begin{Def}
	A \textbf{smooth action} of a diffeological group $ G $ on a diffeological space $ X $ is a smooth
	homomorphism $ G \rightarrow\mathrm{Diff}(X) $.
\end{Def}

\begin{definition}
	A \textbf{diffeological vector space} is a vector space equipped with a so-called \textbf{vector space diffeology} such that the addition and the scalar multiplication are smooth.
\end{definition}

\begin{definition}\cite{CW}
	Let $ X $ be a diffeological space. A \textbf{diffeological vector space over $ X $}  is a smooth map $ f:V\rightarrow X $ and a real vector space structure on each of the fiber $ V_x:=f^{-1}(x) $
	such that the addition map $ V\times_XV\rightarrow V $, the scalar multiplication map $ \mathbb{R}\times V\rightarrow V $, and the zero section $ X\rightarrow V $ are smooth.
\end{definition}

\begin{definition}(\cite{ADD})
	A \textbf{diffeological affine space} is
	a triple
	$ (E,\overrightarrow{E},+) $
	consisting of a non-empty diffeological space $ E $, a diffeological vector space $ \overrightarrow{E} $  and a smooth action
	$ +:E\times\overrightarrow{E}\rightarrow E $, such that
	for any two points $ a, b \in E $, there is a unique $ u\in\overrightarrow{E} $ such that
	$ a + u = b $.
\end{definition}

\subsection{Diffeological fiber bundles}

\paragraph{\textbf{Smooth projections.}}
Any surjective smooth map $ \pi:E\rightarrow X $ is called a \textit{smooth projection}.
For any $ x\in X $, $ \pi^{-1}(x) $  denoted by $ E_x $ is the \textit{fiber} of $ \pi $ over $x$. 
A morphism between smooth projections $ \pi:E\rightarrow X $ and $ \pi':E' \rightarrow X' $ is a pair of smooth maps $ \phi:E\rightarrow E' $ and $ \Pr(\phi):X\rightarrow X' $ such that the following diagram commutes. 
\begin{displaymath}
	\xymatrix{
		E \ar[r]^{\phi}\ar[d]_{\pi} & E' \ar[d]^{\pi'} \\
		X \ar[r]_{\Pr(\phi)} & X'  }
\end{displaymath}
Smooth projections and the morphisms between them make up a category.
Smooth projections $\pi$ and $\pi'$ with the same base $ X $ are 
$ X $-equivalent if there is an isomorphism 
$ (\phi,\mathrm{id}_X) $ between them.

\paragraph{\textbf{Pullbacks.}}
Pullbacks in the category $ \mathsf{Diff} $ of diffeological spaces exist.
This construction can be concretely made as follows.
Let $ f:X\rightarrow Y $ and $ g:Z\rightarrow Y $ be two smooth maps.
The \textbf{pullback} of $ f $ by $ g $ denoted by $ g^*f:g^*X\rightarrow Z $ is a smooth map defined on 
\begin{center}
	$ g^*X=\lbrace(z,x)\in Z\times X\mid g(z)=f(x)\rbrace $
\end{center}
as a subspace of $ Z\times X $, taking $ (z,x) $ to $ z $, that is, $ g^*f $ is the restriction of the projection $ \Pr_1:Z\times X\rightarrow Z $ to $ g^*X $.
This 
gives rise to a natural morphism $ (g_{\#},g) $
\begin{displaymath}
	\xymatrix{
		g^*X \ar[r]^{g_{\#}}\ar[d]_{g^*f} & X\ar[d]^{f} \\
		Z \ar[r]^{g} & Y  }
\end{displaymath}
from $ g^*f $ to $ f $ in the category $ \mathrm{Mor}(\mathsf{Diff}) $ of morphisms of $ \mathsf{Diff} $ (see, e.g., \cite{M}), where $ g_{\#}=\Pr_2|_{g^*X} $ and $ \Pr_2:Z\times X\rightarrow X $ is the projection on the second factor.
It is clear that for each $ z\in Z $, the fiber of $ g^*f $ over $ z $ is diffeomorphic to the fiber of $ f $ over $ g(z) $.
Symmetrically, $ g_{\#} $ plays the role of the pullback of $ g $ by $ f $.

One important property of pullbacks is transitivity.
If $ h:W\rightarrow Z $ is another smooth map, then the pullback of $ g^*f $ by $ h $, i.e., $ h^*g^*f $, is equivalent to the pullback of $ f $ by $ g\circ h $, 
i.e., $  (g\circ h)^*f $ in $ \mathrm{Mor}(\mathsf{Diff}) $, and that $ (g\circ h)_{\#}=g_{\#}\circ h_{\#} $.

If $ E $ is a diffeological $ G $-space,
then $ G $ has a natural action on $ f^*E $ defined by 
$ g_{f^*E}(x',\xi)=(x',g_{E}(\xi)) $ for $ g\in G $, where $ g_{E} $ denotes the action of $ G $ on $ E $. 

\paragraph{\textbf{Local triviality.}}  
\begin{enumerate}
	\item[$ {\tiny\blacktriangleright}  $] 
	Let	$ T $ be a diffeological space.
	A smooth projection $ \pi:E\rightarrow X $ is \textbf{trivial} with respect to the fiber $ T $ if $ \pi $ is $ X $-equivalent to the first factor projection $ \Pr_1: X \times T\rightarrow X $.
	A smooth projection $ \pi:E\rightarrow X $ is \textbf{locally trivial} with respect to the fiber $ T $ if there exists a D-open cover $ \lbrace U_i\rbrace_{i\in J} $ of $ X $ such that 
	the pullback $ \imath_i^*\pi $ by the inclusion $ \imath_i:U_i\hookrightarrow X $ is trivial with respect to $ T $, for every $ i\in J $.
	
	\item[$ {\tiny\blacktriangleright}  $] 
	Let a diffeological group $ G $ act smoothly on a diffeological space $ E $.
	A smooth projection $ \pi:E\rightarrow X $  is a \textbf{trivial} principal $ G $-bundle
	if there exists a $ G $-equivariant diffeomorphism $ \phi:E\rightarrow X\times G $ with $ \Pr_1\circ\phi=\pi $.
	A smooth projection $ \pi:E\rightarrow X $ is
	\textbf{locally trivial} principal $ G $-bundle if there exists a D-open cover
	$ \lbrace U_i\rbrace_{i\in J} $ of $ X $ such that the pullback $ \imath_i^*\pi $ by the inclusion $ \imath_i:U_i\hookrightarrow X $ is a trivial principal $ G $-bundle for each $ i $.   
	
	\item[$ {\tiny\blacktriangleright}  $] 
	Let $ V $ be a diffeological vector space. A diffeological vector space $ \pi:E\rightarrow X $
	over the space $ X $ is \textbf{trivial} of fiber type $ V $ if there exists a diffeomorphism $ \phi:E\rightarrow X\times V $
	with $ \Pr_1\circ\phi=\pi $, such that for every $ x\in X $, the restriction $ \phi|_{\pi^{-1}(x) }:\pi^{-1}(x) \rightarrow V $ is an isomorphism of
	diffeological vector spaces.
	A diffeological vector space $ \pi:E\rightarrow X $ over the space $ X $ is \textbf{locally trivial} of fiber type $ V $
	if there exists a D-open cover 	$ \lbrace U_i\rbrace_{i\in J} $ of $ X $  such that for each $ i $, the pullback $ \imath_i^*\pi $ by the inclusion $ \imath_i:U_i\hookrightarrow X $ is trivial of fiber type $ V $.
\end{enumerate}

\begin{definition}\label{d1}
	(\cite[art. 8.9]{PIZ2013}).
	A \textbf{diffeological fiber bundle} of fiber type $ T $ is a smooth projection $ \pi:E\rightarrow X $ locally trivial along the plots in $ X $, that is,  the pullback of $ \pi $ by every plot in $ X $ is 
	locally trivial with the fiber $ T $.  
\end{definition}

\begin{definition}
	A \textbf{morphism between diffeological fiber bundles} $ \pi:E\rightarrow X $ and $ \pi':E'\rightarrow X' $ 
	is a fiber preserving smooth map $ \phi:E\rightarrow E' $, that is, if 
	$ \pi(\xi_1)=\pi(\xi_2) $, then $ \pi'\circ\phi(\xi_1)=\pi'\circ\phi(\xi_2) $ for every
	$ \xi_1,\xi_2\in E $.
	
\end{definition}
In this situation, $ \phi $ induces a smooth map $ \Pr(\phi):X\rightarrow X' $ between base spaces so that $ (\phi,\Pr(\phi)) $ is a morphism of smooth projections. 
Denoted by $ \mathrm{Bund}_{fiber}(X,T) $ the set of isomorphism classe of diffeological fiber bundles on $ X $ with typical fiber $ T $.

\begin{definition}
	Let a diffeological group $ G $ act smoothly on a diffeological space $ E $.
	A \textbf{diffeological principal $ G $-bundle} is a smooth projection $ \pi:E\rightarrow X $ such that the pullback of $ \pi $ along every plot in $ X $ is a locally trivial $ G $-bundle.
\end{definition}

\begin{definition}
	A morphism between diffeological principal $ G $-bundles $ \pi:E\rightarrow X $ and $ \pi':E'\rightarrow X' $ is
	a $ G $-equivariant smooth map $ \phi:E\rightarrow E' $.
	Two principal $ G $-bundles $ \pi:E\rightarrow X $ and $ \pi':E'\rightarrow X' $  are isomorphic if there exists a $ G $-equivariant diffeomorphism $ \phi:E\rightarrow E' $ such that $ \pi'\circ\phi=\pi $.
\end{definition}

\begin{definition}
	Let $ V $ be a diffeological vector space. 
	A \textbf{diffeological vector bundle} of fiber type $ V $ is a diffeological vector space $ \pi:E\rightarrow X $ over $ X $ such that the pullback of $ \pi $ along every plot in $ X $ is a locally trivial of fiber type $ V $.
\end{definition}

\begin{definition}
	A morphism between diffeological vector bundles $ \pi:E\rightarrow X $ and $ \pi':E'\rightarrow X' $ is
	a fiber preserving smooth map $ \phi:E\rightarrow E' $ such that $ \phi|_{}:E\rightarrow E' $.
\end{definition}

\begin{proposition}\label{P1}
	Let $ X $ be a diffeological space and $ \mathcal{C}\in\mathsf{CGF}(X) $. Assume that $ \pi:E\rightarrow X $ is a smooth projection (resp. smooth projection from a diffeological $ G $-space $ E $, diffeological vector space over $ X $).
	Then $ \pi $ is a diffeological fiber bundle of fiber type $ T $ (resp. diffeological principal $ G $-bundle, diffeological vector bundle of fiber type $ V $) if and only if the pullback of $ \pi $ by every plot $ P $ belonging to $ \mathcal{C} $ is 
	locally trivial of fiber type $ T $ (resp. locally trivial principal $ G $-bundle, locally trivial of fiber type $ V $).  
\end{proposition}
\begin{proof}
	The proof is straightforward.
\end{proof}

\section{Sheaves on diffeological spaces}\label{S3}
In this section, we give a brief overview of sheaves on diffeological spaces and related constructions from  \cite{DA2018,DA}.
\paragraph{\textbf{Site of plots.}}
The category of the plots in a diffeological space $ X $, denoted by $ \mathsf{Plots}(X) $,
has the plots in $ X $ for objects and 
a morphism $ Q\stackrel{F}{\longrightarrow} P $ 
between two plots $ P:U\rightarrow X $ and $ Q:V\rightarrow X $ is a commutative triangle
$$
\xymatrix{
	& X  & \\
	V \ar[ur]^{Q} \ar[rr]_{F} & & U \ar[ul]_{P}
}
$$
where $ F $ is a smooth map between domains.
If $P':U'\rightarrow X$ is a restriction of $ P:U\rightarrow X $, that is, $ U'\subseteq U $ and $ P|_{U'}=P' $,  
one has the inclusion morphism $P'\stackrel{\imath}{\longrightarrow} P$ given by the inclusion $\imath:U'\hookrightarrow U$.
In particular, for a compatible family 
$\lbrace P_i\rbrace_{i\in J}$ 
of plots with the supremum $P$, there are inclusion morphisms $P_i\stackrel{\imath_i}{\longrightarrow} P$.
Endow the category of the plots in a diffeological space $ X $ with the Grothendieck pretopology in which 
a covering for a plot $ P $ is a compatible family of plots with the supremum $ P $.
This site is called the \textbf{site of plots} in $ X $ and denoted by $ X_{\mathsf{Plots}} $.

\begin{definition}
	A \textbf{presheaf} $ A $ of abelian groups on a diffeological space $ X $ is a functor
	$ A:\mathsf{Plots}(X)^{op}\rightarrow\mathsf{Ab}$.
	We denote the corresponding map to $Q\stackrel{F}{\longrightarrow}P$ by $F^*:A(P)\rightarrow A(Q)$ and call it the \textbf{pullback} by $ F $.
	If $P'\stackrel{\imath}{\longrightarrow} P$ is an inclusion morphism, we denote by $s\upharpoonright{P'}$, the pullback of $ s $ by $ \imath $, for every $s\in A(P)$. 
\end{definition}
\begin{definition}
	We denote the limit of a presheaf $ A $ on a diffeological space $X$  by $\Sigma A(X)$ and call it the \textbf{sections} of $A$.
\end{definition}
\begin{definition}
	A \textbf{sheaf} $A$ of abelian groups on diffeological space $X$ is a sheaf on the site $ \mathsf{Plots}(X) $,
	meaning that the sequence
	\begin{equation*}
		A(P)\stackrel{}{\longrightarrow} \prod_{i\in J}A(P_i)\stackrel{}{\longrightarrow}\displaystyle\prod_{(i,j)\in J\times J}A(P_i\times_PP_j)
	\end{equation*}
	is exact, for any plot $ P $ and compatible family $\lbrace P_i\rbrace_{i\in J}$ of plots with the supremum $P$.
\end{definition}
\begin{example}\label{exa-shv-g}
	Let $ X $ be a diffeological space and $ G $ be a diffeological group.
	Assignment to each plot $ P $ in $ X $ the group $ G(P) $ of $ G $-valued smooth functions defined on $ \mathrm{dom}(P) $ and to each morphism $ Q\stackrel{F}{\longrightarrow} P $, the pullback homomorphism $ F^*:G(P)\rightarrow G(Q) $ taking $ f $ to $ f\circ F $, defines a sheaf on $ X $.
\end{example}

\begin{theorem}
	For a presheaf $ A $ on a diffeological space $X$,
	consider the presheaf $\Sigma A$ on the D-topological space $X$ by assigning 
	the set $\Sigma A|_U(U)$ of sections of $A|_U$ to each D-open subspace $U$ of $X$
	along with the restrictions of sections. 
	If $A$ is a sheaf on  $X$, then $\Sigma A$ is a sheaf on the D-topological space $X$. 
\end{theorem}

Let $ X $ be a diffeological space and $ \mathcal{C}\in\mathsf{CGF}(X) $. Denote by $\Sigma A(\mathcal{C})$ the limit of the restriction of a presheaf $ A $ on $ X $ to the prediffeology $\lfloor\mathcal{C}\rfloor$.
By universal property, there is a canonical morphism $ \alpha_{\mathcal{C}}:\Sigma A(X)\rightarrow\Sigma S(\mathcal{C}) $.
\begin{definition}
	Let $ X $ be a diffeological space.
	A \textbf{quasi-sheaf} on $ X $ is a presheaf $ A $ such that the canonical morphism
	$ \alpha_{\mathcal{C}}:\Sigma A(X)\rightarrow\Sigma A(\mathcal{C}) $
	is an isomorphism, for every $ \mathcal{C}\in\mathsf{CGF}(X) $.
	In other words, for every covering generating family $ \mathcal{C}$ of $ X $, the following holds:
	\begin{enumerate}
		\item[\bf{QS.}]   
		For every $ \sigma\in\Sigma A(\mathcal{C}) $ there exists a unique \textbf{extension} $ \overline{\sigma}\in\Sigma A(X) $ of $ \sigma $, i.e., $ \overline{\sigma}(P)=\sigma(P) $ for all $ P\in\mathcal{C} $.
	\end{enumerate}
\end{definition}
\begin{proposition}
	A presheaf $ A $ on a a diffeological space $ X $ is a quasi-sheaf if and only if for any covering generating family $ \mathcal{C}$ of $ X $, the sequence
	\begin{equation*}
		\Sigma A(X)\stackrel{}{\longrightarrow} \prod_{P\in\mathcal{C}}A(P)\stackrel{}{\longrightarrow}
		\displaystyle
		\prod_{P\stackrel{F}{\longleftarrow}Q\stackrel{F'}{\longrightarrow} P'}
		A(P\stackrel{F}{\longleftarrow}Q\stackrel{F'}{\longrightarrow} P')
	\end{equation*}
	is exact,
	where 
	the latter product is on morphisms $ P\stackrel{F}{\longleftarrow}Q\stackrel{F'}{\longrightarrow} P' $ of plots in $ X $ with $ P,P'\in \mathcal{C} $ and $ A(P\stackrel{F}{\longleftarrow}Q\stackrel{F'}{\longrightarrow} P') $ is equal to $ A(Q) $.
\end{proposition}

\begin{theorem}
	Every sheaf on a diffeological spaces is a quasi-sheaf.
\end{theorem}

\subsection{A framework for cohomology theories based on presheaves}\label{S33}
Here, we discuss a general framework based on presheaves
to exhibit some cohomology theories of diffeological spaces (see \cite{DA2018}).

\begin{definition}
	
	A \textbf{cochain complex} $(A^{\bullet},d)$ of presheaves of abelian groups on a diffeological space $X$
	is a sequence of presheaves on $X$ and morphisms
	\begin{center}
		$\cdots\longrightarrow A^{k-1}\stackrel{d}{\longrightarrow}A^{k}\stackrel{d}{\longrightarrow}A^{k+1}\longrightarrow\cdots$
	\end{center}
	such that the composition of any two consecutive morphisms is zero, i.e., $d\circ d=0$.
	The morphism $d^{k}:A^{k}\rightarrow A^{k+1}$ is called the \textbf{$k$th coboundary operator}.
\end{definition}

In this situation, we obtain cochain complexes $(A^{\bullet}(P),d_P)$ of abelian groups for all plots $P$ in $ X $.
Because coboundary operators are natural transformation, 
a morphism  $ Q\stackrel{F}{\longrightarrow} P $ of plots induces a well-defined homomorphism
$ (A^{\bullet}(Q),d_Q)\rightarrow(A^{\bullet}(P),d_P) $ of cochain complexes.
Therefore, the assignment 
\begin{align*}
	P & \quad\rightsquigarrow\quad H^k(A^{\bullet}(P)) \\
	Q\stackrel{F}{\longrightarrow} P & \quad\rightsquigarrow\quad H^k(A^{\bullet}(P))\stackrel{F^{*}}{\longrightarrow}  H^k(A^{\bullet}(Q))  
\end{align*}
in which  $ H^k(A^{\bullet}(P)) $ is the $k$th cohomology group of the cochain complex $(A^{\bullet}(P),d_P)$,
defines a presheaf, denoted by $ \textbf{H}^k(A^{\bullet}) $ and called the \textbf{$k$th cohomology presheaf} of the cochain complex $(A^{\bullet},d)$, 

Moreover, the cochain complex $(A^{\bullet},d)$ 
has an \textbf{associated cochain complex} $(\Sigma A^{\bullet},\Sigma d)$ of groups of sections
\begin{center}
	$\cdots\longrightarrow\Sigma A^{k-1}(X)\stackrel{\Sigma d}{\longrightarrow}\Sigma A^{k}(X)\stackrel{\Sigma d}{\longrightarrow}\Sigma A^{k+1}(X)\longrightarrow\cdots $
\end{center}
whose $k$th associated cohomology group of cochain complex $(A^{\bullet},d)$ is denoted by $H^k(\Sigma  A^{\bullet})$.

By the universal property of limits, there exists a homomorphism 
\begin{center}
	$ \varrho_A:H^k(\Sigma  A^{\bullet}) \longrightarrow\Sigma \textbf{H}^k(A^{\bullet}) $
\end{center}
given by $ \varrho_A([\sigma])(P)= [\sigma(P)]$,
where $ [\sigma] $ and $ [\sigma(P)] $ are the classes of $ \sigma $ and $ \sigma(P) $ in $ H^k(\Sigma  A^{\bullet}) $ and $ \textbf{H}^k(A^{\bullet})(P) $, respectively.

\begin{definition}
	A sequence 
	\begin{center}
		$\cdots\longrightarrow A\stackrel{\Phi}{\longrightarrow}B\stackrel{\Psi}{\longrightarrow}C\longrightarrow\cdots$
	\end{center}
	of sheaves on $X$ and morphisms is \textbf{exact} if and only if for each plot $P$ in $ X $,
	\begin{center}
		$\cdots\longrightarrow A(P)\stackrel{\Phi_P}{\longrightarrow}B(P)\stackrel{\varPsi_P}{\longrightarrow}C(P)\longrightarrow\cdots$
	\end{center}
	is an exact sequence of usual sheaves.
	This means that
	\begin{enumerate}
		\item[$ {\tiny\blacktriangleright}  $] 
		for all $ a\in A(P) $, $ \Psi_P\circ\Phi_P(a)=0 $,
		\item[$ {\tiny\blacktriangleright}  $] 
		if $ \Psi_P(b)=0 $, then for every $ r\in dom(P) $, there exists an open neighborhood $ r\in V\subseteq dom(P) $, such that $ b\upharpoonright{(P|_V)}=\Phi_P(a) $ for some $ a\in A(P|_V) $.
	\end{enumerate}
\end{definition}

\begin{definition}
	A \textbf{resolution} of a sheaf $ \mathcal{Q} $ on a diffeological space $ X $ is an exact sequence of sheaves on $ X $ and morphisms
	\begin{center}
		$0\longrightarrow \mathcal{Q} \longrightarrow A^{0}\stackrel{d_0}{\longrightarrow}A^{1}\stackrel{d_1}{\longrightarrow}A^{2}\longrightarrow\cdots $.
	\end{center}
\end{definition}


\subsection{De Rham cohomology}
Let $ X $ be a diffeological space and for every nonnegative integer $ k $, consider the sheaf $ \Lambda^k $ assigning to each plot $ P $ in $ X $ the abelian group $ \Lambda^k(P) $ of differential $ k $-forms on the domain of $ P $ and to each morphism $ Q\stackrel{F}{\longrightarrow} P $, the pullback homomorphism $ F^*:\Lambda^k(P)\rightarrow\Lambda^k(Q) $ of differential forms \cite[Example 3.9]{DA}.
These sheaves along with coboundary operators $ d^k:\Lambda^k\rightarrow\Lambda^{k+1} $  of the exterior derivatives
$ d^k_P:\Lambda^k(P)\rightarrow\Lambda^{k+1}(P) $
of differential forms on the domain of plots construct a cochain complex $ (\Lambda^{\bullet},d) $.

\begin{proposition}
	
	The associated cochain complex $(\Sigma\Lambda^{\bullet},\Sigma d)$ of groups of sections is the same as the cochain complex of differential forms on a diffeological space $X$. 
	Therefore, the associated cohomology 
	$H^k(\Sigma\Lambda^{\bullet}) $ coincides with the de Rham cohomology $ H^k_{dR}(X) $ of the space $X$.
\end{proposition}
\begin{proof}
	The proof is straightforward.
\end{proof}

\begin{proposition}\label{prop-res}
	The sequence
	\begin{center}
		$0\longrightarrow\underline{\mathbb{R}}\stackrel{}{\longrightarrow}\Lambda^{0}\stackrel{d}{\longrightarrow}\Lambda^{1}\stackrel{d}{\longrightarrow}\Lambda^{2}\longrightarrow\cdots $
	\end{center}
	is a resolution of the constant sheaf $ \underline{\mathbb{R}} $ assigning to every plot $ P $ in $ X $ the group of locally constant functions on $ \mathrm{dom}(P) $ with values in $ \mathbb{R} $.
\end{proposition}
\begin{proof}
	Indeed, for each plot $ P $ in $ X $, the sequence 
	\begin{center}
		$0\longrightarrow\underline{\mathbb{R}}(P)\stackrel{}{\longrightarrow}\Lambda^{0}(P)\stackrel{d_P}{\longrightarrow}\Lambda^{1}(P)\stackrel{d_P}{\longrightarrow}\Lambda^{2}(P)\longrightarrow\cdots $
	\end{center}
	is an exact sequence of usual sheaves on $ \mathrm{dom}(P) $.
\end{proof}

\section{Diffeological \v{C}ech cohomology}\label{S4}

Now we introduce diffeological \v{C}ech cohomology with coefficients in a (pre)sheaf, formulated in terms of covering generating families.
Assume that $A$ is a presheaf of abelian groups on a diffeological space $ X $ and $\mathcal{C}$ is a covering generating family of $ X $.
\begin{definition}(\cite{AD}). 
	We define $ n $-simplices on $ \mathcal{C} $ inductively:
	\begin{enumerate}
		\item[$ (i) $]
		A $ 0 $-\textbf{simplex} on $ \mathcal{C} $ is just an element $ P_0 $ of $\mathcal{C}$.
		The nerve plot of a $0$-simplex $ P_0 $ is the plot $ P_0 $ itself, by convention.
		\item[$ (ii) $]
		A $ 1 $-\textbf{simplex} on $ \mathcal{C} $ is any diagram 
		\begin{center}
			$ P_0\stackrel{F_1}{\longleftarrow}Q\stackrel{F_0}{\longrightarrow} P_1 $ 
		\end{center}
		with $P_0, P_1\in\mathcal{C}$. The plot $ Q $ is called the nerve plot.
		Notice that $ Q $ is the nerve plot of the diagram not that of $ P_0, P_1 $.
		
		\item[$ (iii) $]
		For integers $ n\geqslant 2 $, 
		an $ n $-\textbf{simplex} $ (P_0,\ldots,P_n) $ on $ \mathcal{C} $ consists of $ n+1 $ plots $ P_0,\ldots,P_n $ belonging to $\mathcal{C}$ and a \textbf{nerve plot} $ Q $ in $ X $ such that any $ n $ plots 
		$ P_0,\ldots,\widehat{P_i},\ldots,P_n $
		(the hat indicates the omission of $ P_i $) form an $ (n-1) $-simplex with the nerve plot $ Q_i $.
		In addition, 
		for each $ i=0,\ldots,n $, there exist a morphism $ Q\stackrel{F_i}{\longrightarrow} Q_i $ commuting with the morphisms $ Q_i\stackrel{F_{i,j}}{\longrightarrow} Q_{i,j} $, for $ (n-2) $-simplices 
		\begin{center}
			$ P_0,\ldots,\widehat{P_i},\ldots,\widehat{P_j},\ldots,P_n $ 
		\end{center}
		with the nerve plots $ Q_{i,j} $; that is, 
		$ F_{i,j}\circ F_i= F_{j,i}\circ F_j $.
		
	\end{enumerate}
\end{definition}

For example,  the following commutative diagram is a $2$-simplex on $ \mathcal{C} $:
\begin{center}
	\begin{tikzpicture}
		\matrix (m) [matrix of math nodes, row sep=.8em,
		column sep=0.6em, text height=1.2ex, text depth=0.25ex]
		{ & & P_0 & & \\
			& &  & & \\
			& &  & & \\
			& Q_2  & & Q_1 & \\
			& & Q & & \\
			P_1 & & Q_0 & & P_2 \\};
		\path[->,font=\scriptsize]
		(m-4-2) edge node[auto] {$ F_{2,1} $} (m-1-3) edge node[above] {$F_{2,0}~~~~~~~$} (m-6-1)
		(m-4-4) edge node[above] {$~~~~~~~F_{1,2} $} (m-1-3) edge node[auto] {$F_{1,0} $} (m-6-5)
		(m-6-3) edge node[auto] {$ F_{0,2} $} (m-6-1) edge node[below] {$F_{0,1} $} (m-6-5)
		(m-5-3) edge node[auto]  {$ F_2 $} (m-4-2) edge node[below]  {$~~~~~F_1 $} (m-4-4) edge node[auto]  {$ F_0 $} (m-6-3);
	\end{tikzpicture} 
\end{center}

For an $n$-simplex $ (P_0,\ldots,P_n) $ on $ \mathcal{C} $ with the nerve plot $ Q $, 
let $ A(P_0,\ldots,P_n)=A(Q) $, and set 
\begin{center}
	$  C^n(X,\mathcal{C},A)= \displaystyle\prod_{n-\mathrm{simplexes}(\mathcal{C})} A(P_0,\ldots,P_n)$,
\end{center}
the group of $ n $-\textbf{cochains} with coefficients in the presheaf $ A $ subordinated to the covering generating family $\mathcal{C}$.
Define the operators 
$ \delta^n:C^n(X,\mathcal{C},A)\longrightarrow C^{n+1}(X,\mathcal{C},A)$ 
for nonnegative integers $ n $ by
\begin{center}
	$ \delta^n(c)(P_0,\ldots,P_{n+1})=\displaystyle\sum_{i=0}^{n+1}(-1)^i~F^*_i c (P_0,\ldots,\widehat{P_i},\ldots,P_{n+1})$. 
\end{center}

\begin{proposition}
	
	$ \delta^{n+1}\circ \delta^n=0 $, for every nonnegative integer $ n $.
\end{proposition}
\begin{proof}
	For any $ n $-cochain $ c $ and $(n+2)$-simplex $ (P_0,\ldots,P_{n+2}) $ we have
	\begin{align*}
		(\delta^{n+1}\circ \delta^n)(c)(P_0,\ldots,P_{n+2})&=\displaystyle\sum_{i=0}^{n+2}(-1)^i F^*_i \delta^n(c) (P_0,\ldots,\widehat{P_i},\ldots,P_{n+2})\\
		&=\displaystyle\sum_{i=1}^{n+2}\sum_{j=0}^{i-1}(-1)^{i+j} (F^*_i\circ F^*_{i,j}) c (P_0,\ldots,\widehat{P_j},\ldots,\widehat{P_i},\ldots,P_{n+2})\\
		&+\displaystyle\sum_{i=0}^{n+1}\sum_{j=i+1}^{n+2}(-1)^{i+j-1} (F^*_i\circ F^*_{i,j}) c (P_0,\ldots,\widehat{P_i},\ldots,\widehat{P_j},\ldots,P_{n+2})\\
		&=\displaystyle\sum_{i=1}^{n+2}\sum_{j=0}^{i-1}(-1)^{i+j} (F^*_i\circ F^*_{i,j}) c (P_0,\ldots,\widehat{P_j},\ldots,\widehat{P_i},\ldots,P_{n+2})\\
		&-\displaystyle\sum_{j=0}^{n+1}\sum_{i=j+1}^{n+2}(-1)^{i+j} (F^*_j\circ F^*_{j,i}) c (P_0,\ldots,\widehat{P_j},\ldots,\widehat{P_i},\ldots,P_{n+2})=0.
	\end{align*}
\end{proof}

Thus, we obtain a cochain complex $ C^{\bullet}(X,\mathcal{C},A) $,
\begin{center}
	$ 0\longrightarrow C^0(X,\mathcal{C},A)\stackrel{\delta^0}{\longrightarrow} C^1(X,\mathcal{C},A)\longrightarrow\cdots\longrightarrow C^n(X,\mathcal{C},A)\stackrel{\delta^n}{\longrightarrow}C^{n+1}(X,\mathcal{C},A)\longrightarrow\cdots$.
\end{center}

\begin{notation}
	We denote the $n$th cohomology group of the cochain complex $ C^{\bullet}(X,\mathcal{C},A) $ by $ H^n(X,\mathcal{C};A)$. 
\end{notation}

\begin{definition}(\cite{AD}). 
	Let  $X$ be a diffeological space and $ \mathcal{C}=\{P_\alpha\}_{\alpha\in I} $ be a covering generating family of $ X $.
	A \textbf{refinement} of $ \mathcal{C} $ is a covering generating family $ \mathcal{C}'=\{P'_\beta\}_{\beta\in J} $ together with 
	a map $ \lambda:J\rightarrow I $ and a family $ \{f_{\beta}\}_{\beta\in J} $  of morphisms $ P'_\beta\stackrel{f_\beta}{\longrightarrow} P_{\lambda(\beta)} $.
	In this situation, we have $ \mathcal{C}'\subseteq\lfloor\mathcal{C}\rfloor $ and denote such a refining by $ \mathcal{C}'\prec\mathcal{C} $.
	Refinements of covering generating families of $ X $ turn $ \mathsf{CGF}(X) $ into a category in a trivial way.
	A \textbf{common refinement} of $ \mathcal{C} $ and $ \mathcal{C}' $ is a covering generating family $ \mathcal{B} $ such that
	$ \mathcal{B}\prec\mathcal{C} $ and $ \mathcal{B}\prec\mathcal{C}' $.
\end{definition}

If  $ \mathcal{C}'\prec\mathcal{C} $  along with $ \lambda:J\rightarrow I $ is a refinement and
$ \sigma'=(P'_{\beta_0},\ldots,P'_{\beta_n}) $ is an $ n $-simplex on $ \mathcal{C}' $ with the nerve plot $ Q $, 
then $ \lambda(\sigma')=(P_{\lambda(\beta_0)},\ldots,P_{\lambda(\beta_n)})  $ constitutes an $ n $-simplex on $ \mathcal{C} $ with the nerve plot $ Q $. 
Moreover, $ \lambda:J\rightarrow I $  defines a chain morphism 
$ \lambda^{*}:C^{\bullet}(X,\mathcal{C},A)\rightarrow C^{\bullet}(X,\mathcal{C}',A)$
by 
\begin{center}
	$ \lambda^{*}_0(c)(P'_\beta)=f^{*}_{\beta}c(P_{\lambda(\beta)}) $
\end{center}
and
\begin{center}
	$ \lambda^{*}_n(c)(P'_{\beta_0},\ldots,P'_{\beta_n})=c(P_{\lambda(\beta_0)},\ldots,P_{\lambda(\beta_n)}),\qquad $ for $ n\geq 1 $.
\end{center}
Then	$ \delta^n\circ \lambda^{*}_n=\lambda^{*}_{n+1} \circ \delta^n $ for all $ n\geq 0 $.
Let $ \mu:J\rightarrow I $  be another refinement map of  $ \mathcal{C}'\prec\mathcal{C} $. Consider
$ D_n:C^{n}(X,\mathcal{C},A)\rightarrow C^{n-1}(X,\mathcal{C}',A) $ given by
\begin{center}
	$ D_n(c)(P'_{\beta_0},\ldots,P'_{\beta_{n-1}})=\displaystyle\sum_{j=0}^{n-1} (-1)^jc(P_{\mu(\beta_0)},\ldots,P_{\mu(\beta_j)},P_{\lambda(\beta_j)},\ldots,P_{\lambda(\beta_{n-1})}) $.
\end{center}
A straightforward computation shows that
\begin{center}	
	$ D_{n+1}\circ \delta^n-\delta^n\circ D_n=\lambda^{*}-\mu^{*} $.
\end{center}
That is, $ D $ is a chain homotopy between $ \lambda^{*} $ and $ \mu^{*} $. Therefore, any refinement $ \mathcal{C}'\prec\mathcal{C} $ induces a homomorphism 
$ j(\mathcal{C}',\mathcal{C}):H^n(X,\mathcal{C};A)\rightarrow H^n(X,\mathcal{C}';A) $
independent of the particular refinement map. 
Moreover, for refinements
$ \mathcal{C}''\prec\mathcal{C}'\prec\mathcal{C} $
we have
\begin{center}
	$ j(\mathcal{C}'',\mathcal{C})= j(\mathcal{C}',\mathcal{C})\circ j(\mathcal{C}'',\mathcal{C}') $.
\end{center}
This gives rise to a functor 
$ H^n(X,-;A):\mathsf{CGF}(X)^{op}\rightarrow\mathsf{Ab} $,  for each $ n\geq 0 $.
\begin{definition} 
	We define the $ n $-th \textbf{diffeological \v{C}ech cohomology group} of $ X $ with coefficients in a presheaf $ A $ as
	\begin{center}
		$ \check{H}^n(X;A)=\varinjlim_{\mathcal{C}}H^n(X,\mathcal{C};A). $
	\end{center}
\end{definition}
Naturally, $ H^0(X,\mathcal{C};A)$ is isomorphic to $ \ker\delta^0 $, i.e.,  
\begin{center}
	$\ker\Big{(}~\displaystyle\prod_{P\in\mathcal{C}}A(P)\stackrel{}{\longrightarrow}\displaystyle\prod_{P_0\stackrel{F_0}{\longleftarrow}Q\stackrel{F_1}{\longrightarrow} P_1}A(P_0\stackrel{F_0}{\longleftarrow}Q\stackrel{F_1}{\longrightarrow} P_1)~\Big{)} $, 
\end{center}
which is the same as $ \Sigma A(\mathcal{C}) $. 
Now consider the commutative diagram
\begin{displaymath}
	\xymatrix{
		\Sigma A(X)\ar[r]^{\alpha_{\mathcal{C}}}\ar[d]^{} &\Sigma A(\mathcal{C})\ar[r]^{\beta_{\mathcal{C}}\quad}\ar[d]^{} &  H^0(X,\mathcal{C};A)\ar[d]^{j(\mathcal{C}',\mathcal{C})}\ar[r]^{\gamma_{\mathcal{C}}} & \check{H}^0(X;A)\ar[d]^{}\\
		\Sigma A(X)\ar[r]^{\alpha_{\mathcal{C}'}}&\Sigma A(\mathcal{C}')\ar[r]^{\beta_{\mathcal{C}'}\quad}& H^0(X,\mathcal{C}';A)\ar[r]^{\gamma_{\mathcal{C}'}}& \check{H}^0(X;A)
	}
\end{displaymath}
where $ \gamma_{\mathcal{C}}:H^n(X,\mathcal{C};A)\rightarrow\check{H}^n(X;A) $ denotes the homomorphism given by the definition of colimit for $ \mathcal{C}\in\mathsf{CGF}(X) $. 
\begin{proposition}\label{p-0th}
	If $ A $ is a quasi-sheaf (and especially, a sheaf), then the natural homomorphism $ h=\gamma_{\mathcal{C}}\circ\beta_{\mathcal{C}}\circ\alpha_{\mathcal{C}}:\Sigma A(X)\longrightarrow \check{H}^0(X;A) $ given above is an isomorphism.
\end{proposition}
\begin{proof}	
	If $ A $ is a quasi-sheaf,  $ \alpha_{\mathcal{C}}:\Sigma A(X)\rightarrow \Sigma A(\mathcal{C}) $ is an isomorphism for all $ \mathcal{C}\in\mathsf{CGF}(X) $. 
	Then 
	\begin{displaymath}
		\xymatrix{
			&\Sigma A(X)  & \\
			H^0(X,\mathcal{C};A) \ar[ur]^{(\beta_{\mathcal{C}}\circ\alpha_{\mathcal{C}})^{-1}}\ar[rr]_{j(\mathcal{C}',\mathcal{C})}& & H^0(X,\mathcal{C}';A) \ar[ul]_{(\beta_{\mathcal{C}'}\circ\alpha_{\mathcal{C}'})^{-1}}
		}
	\end{displaymath}
	is a cocone. The universal property gives us a unique homomorphism $ g:\check{H}^0(X;A)\rightarrow \Sigma A(X) $ with $ g\circ\gamma_{\mathcal{C}}=(\beta_{\mathcal{C}}\circ\alpha_{\mathcal{C}})^{-1} $. 
	Obviously,
	\begin{center}
		$ g\circ(\gamma_{\mathcal{C}}\circ\beta_{\mathcal{C}}\circ\alpha_{\mathcal{C}})=g\circ h=\mathrm{id} $.
	\end{center}
	On the other hand, we have
	\begin{center}
		$ h\circ g \circ\gamma_{\mathcal{C}}=(\gamma_{\mathcal{C}}\circ\beta_{\mathcal{C}}\circ\alpha_{\mathcal{C}})\circ(\beta_{\mathcal{C}}\circ\alpha_{\mathcal{C}})^{-1}=\gamma_{\mathcal{C}}. $
	\end{center}
	which that 
	$ h\circ g=\mathrm{id} $, because the identity $ \check{H}^0(X;A)\rightarrow\check{H}^0(X;A) $ is the only homomorphism with this property by universal property.
	Therefore, $ h $  is an isomorphism.
\end{proof}

Given a morphism $ \phi:A\rightarrow A' $ of presheaves on $X$,
define 
$ \phi^n_{*}:C^n(X,\mathcal{C},A)\rightarrow C^n(X,\mathcal{C},A')$
by 
$ \phi^n_{*}(c)(P_0,\ldots,P_n)=\phi_{Q}\big{(}c(P_0,\ldots,P_n)\big{)}  $,
where $Q$ is the nerve plot of the $ n $-simplex $ (P_0,\ldots,P_n)$.
Then we have
$ \delta^n\circ \phi^n_{*}=\phi^{n+1}_{*}\circ \delta^n $, so that $ \phi:A\rightarrow A' $
induces  a homomorphism 
$ \phi_{\#}:H^n(X,\mathcal{C};A)\rightarrow H^n(X,\mathcal{C};A')$
between cohomology groups,
for every nonnegative integer $n$.
If $ \lambda:\mathcal{C}'\rightarrow\mathcal{C}  $ is a refinement, then 
the following diagram commutes:
\begin{displaymath}
	\xymatrix{
		H^n(X,\mathcal{C};A) \ar[d]_{\lambda^{*}}\ar[r]^{\phi_{\#}}& H^{n}(X,\mathcal{C};A') \ar[d]^{\lambda^{*}}\\
		H^n(X,\mathcal{C}';A)\ar[r]_{\phi_{\#}} & H^{n}(X,\mathcal{C}';A') 
	}
\end{displaymath}
Hence by universal property, $ \phi:A\rightarrow A' $ induces a  unique homomorphism
\begin{center}
	$ \phi_{\#}:\check{H}^n(X;A)\rightarrow \check{H}^n(X;A')$ 
\end{center}
with the property
\begin{displaymath}
	\xymatrix{
		\check{H}^n(X;A)\ar[r]^{\phi_{\#}}& \check{H}^n(X;A') \\
		H^n(X,\mathcal{C};A)\ar[u]^{\gamma_{\mathcal{C}}}\ar[r]_{\phi_{\#}} & H^{n}(X,\mathcal{C};A') \ar[u]_{\gamma'_{\mathcal{C}}}
	}
\end{displaymath}
where $ \gamma_{\mathcal{C}}:H^n(X,\mathcal{C};A)\rightarrow\check{H}^n(X;A) $ and $ \gamma'_{\mathcal{C}}:H^n(X,\mathcal{C};A')\rightarrow\check{H}^n(X;A') $ denote the homomorphisms given by the definition of colimit for $ \mathcal{C}\in\mathsf{CGF}(X) $. 

We also have
$ (\psi\circ\phi)_{\#}=\psi_{\#}\circ\phi_{\#} $
for morphism $ \phi:A\rightarrow A' $ and $ \psi:A'\rightarrow A'' $ of sheaves on $X$,
and $ \mathrm{id}_{\#}=\mathrm{id} $ for the identity morphism $ \mathrm{id}:A\rightarrow A $.
\begin{proposition}\label{p-long}
	Any short exact sequence
	$ 0\longrightarrow A'\stackrel{\phi}{\longrightarrow} A\stackrel{\psi}{\longrightarrow} A'' \longrightarrow 0 $
	of sheaves on a diffeological space $ X $ induces a long exact sequence
	\begin{align*}
		0\longrightarrow&\Sigma A'(X)\stackrel{\Sigma\phi}{\longrightarrow}\Sigma A(X)
		\stackrel{\Sigma\psi}{\longrightarrow}\Sigma A''(X)\stackrel{}{\longrightarrow}H^{1}(X;A')
		\\
		\cdots	\longrightarrow& \check{H}^n(X;A')
		\stackrel{\phi_{\#}}{\longrightarrow}\check{H}^n(X;A)
		\stackrel{\psi_{\#}}{\longrightarrow}\check{H}^n(X;A'')
		\stackrel{\partial}{\longrightarrow}\check{H}^{n+1}(X;A')
		\longrightarrow \cdots
	\end{align*}
	In addition, for any morphism of short exact sequences of sheaves 
	\begin{displaymath}
		\xymatrix{
			0 \ar[r] & A\ar[r]\ar[d]& A'\ar[r]\ar[d]& A''\ar[r]\ar[d]&0  \\
			0 \ar[r] & B\ar[r] & B'\ar[r]& B''\ar[r]&0 
		}
	\end{displaymath}
	the following diagram commutes:
	\begin{displaymath}
		\xymatrix{
			\check{H}^n(X ;A'') \ar[r]^{\partial}\ar[d]& \check{H}^{n+1}(X ;A) \ar[d]\\
			\check{H}^n(X ;B'') \ar[r]_{\partial} & \check{H}^{n+1}(X ;B) 
		}
	\end{displaymath}
	Hence for each $ n\geq 0 $, we obtain an exact $ \partial $-functor  $ \check{H}^n(X,-):\mathsf{Ab(Shv)}(X)\rightarrow\mathsf{Ab} $.
\end{proposition}
\begin{proof}
	Suppose that
	\begin{center}
		$ 0\longrightarrow A'\stackrel{\phi}{\longrightarrow} A\stackrel{\psi}{\longrightarrow} A'' \longrightarrow 0 $
	\end{center}
	is a short exact sequence of sheaves on a diffeological space $ X $. By left-exactness, for all plots $ P $ in $ X $, sequences
	\begin{center}
		$ 0\longrightarrow A'(P)\stackrel{\phi_{P}}{\longrightarrow} A(P)\stackrel{\psi_{P}}{\longrightarrow} A''(P) $
	\end{center}
	are exact and consequently, so is the sequence
	\begin{center}
		$ 0\longrightarrow C^n(X,\mathcal{C},A')\stackrel{\phi_{*}}{\longrightarrow} C^n(X,\mathcal{C},A)\stackrel{\psi_{*}}{\longrightarrow} C^n(X,\mathcal{C},A''). $
	\end{center}
	Let $ Im_{A''} $ be the presheaf of the image of $ \psi $ as a sub-presheaf of $ A'' $ 
	so that
	\begin{center}
		$ 0\longrightarrow C^n(X,\mathcal{C},A')\stackrel{\phi_{*}}{\longrightarrow} C^n(X,\mathcal{C},A)\stackrel{\psi_{*}}{\longrightarrow} C^n(X,\mathcal{C},Im_{A''})\longrightarrow 0 $
	\end{center}
	is an exact sequence of cochain complexes.
	This gives rise to a long exact sequence
	\begin{align*}
		\cdots	\longrightarrow& H^n(X,\mathcal{C};A')
		\stackrel{\phi_{\#}}{\longrightarrow}H^n(X,\mathcal{C};A)
		\stackrel{\psi_{\#}}{\longrightarrow}H^n(X,\mathcal{C};Im_{A''})
		\stackrel{\partial}{\longrightarrow}H^{n+1}(X,\mathcal{C};A')
		\longrightarrow \cdots
	\end{align*}
	Taking colimits over covering generating families, one obtain  a long exact sequence
	\begin{align}\label{eq1}
		\cdots	\longrightarrow& \check{H}^n(X;A')
		\stackrel{\phi_{\#}}{\longrightarrow}\check{H}^n(X;A)
		\stackrel{\psi_{\#}}{\longrightarrow}\check{H}^n(X;Im_{A''})
		\stackrel{\partial}{\longrightarrow}\check{H}^{n+1}(X;A')
		\longrightarrow \cdots
	\end{align}
	Similarly, by zigzag lemma, for any morphism of short exact sequences of sheaves
	\begin{displaymath}
		\xymatrix{
			0 \ar[r] & A\ar[r]\ar[d]& A'\ar[r]\ar[d]& A''\ar[r]\ar[d]&0  \\
			0 \ar[r] & B\ar[r] & B'\ar[r]& B''\ar[r]&0 
		}
	\end{displaymath}
	the following diagram commutes:
	\begin{displaymath}
		\xymatrix{
			\check{H}^n(X ;Im_{A''}) \ar[r]^{\partial}\ar[d]& \check{H}^{n+1}(X ;A) \ar[d]\\
			\check{H}^n(X ;Im_{B''}) \ar[r]_{\partial} & \check{H}^{n+1}(X ;B) 
		}
	\end{displaymath}
	
	But for each $ n $, the inclusion cochain map
	$  C^n(X,\mathcal{C},Im_{A''})\rightarrow C^n(X,\mathcal{C},A'') $
	induces a natural homomorphism
	$ \theta^n_{\mathcal{C}}:H^n(X,\mathcal{C};Im_{A''}) \rightarrow H^n(X,\mathcal{C};A'') $.
	On passage to the colimit, we get a homomorphism
	\begin{center}
		$ \theta^n_{\#}:\check{H}^n(X;Im_{A''}) \rightarrow \check{H}^n(X;A''). $
	\end{center}
	To complete the proof, it is enough to show that $ \theta^n_{\#} $ is in fact an isomorphism.
	
	To see that $ \theta^n_{\#} $ is surjective, let $ [c'']\in \check{H}^n(X;A'') $ so that $ [c]\in H^n(X,\mathcal{C};A'') $ for some $ \mathcal{C}\in\mathsf{CGF}(X) $.
	Then $ c'' $ is a cocycle element of $ C^n(X,\mathcal{C};A'') $.
	As the sequence
	\begin{center}
		$ 0\longrightarrow A'\stackrel{\phi}{\longrightarrow} A\stackrel{\psi}{\longrightarrow} A'' \longrightarrow 0 $
	\end{center}
	is exact, one can solve locally the equation $ \psi_{P}(a) =a'' $ on each $ P\in\mathcal{C} $ for all $ a''\in A''(P) $.
	In particular, one can find a refinement $ \lambda:\mathcal{G}\rightarrow\mathcal{C}  $, consisting of restrictions of elements of $ \mathcal{C} $, such that for all $ P\in\mathcal{G} $ and $ a''\in A''(P) $, the equation $ \psi_{P}(a) =a'' $ has a solution $ a\in A(P) $.
	Thus, there exists $ c\in C^n(X,\mathcal{G};A) $ such that $ \psi_{*}(c)=\lambda^*c'' $, which implies 
	$ \theta^n_{\#}[\psi_{*}(c)]=[\lambda^*c'']=[c''] $.
	
	To see that $ \theta^n_{\#} $ is injective, let $ \theta^n_{\#}([c])=0 $ for any cocycle element $ c\in C^n(X,\mathcal{C};Im_{A''}) $.
	Then we get $ \lambda^*c=\delta^{n-1}(b) $, for some refinement $ \lambda:\mathcal{G}\rightarrow\mathcal{C}  $ and $ b\in C^{n-1}(X,\mathcal{G};A'') $.
	Again by exactness of
	$ 0\longrightarrow A'\stackrel{\phi}{\longrightarrow} A\stackrel{\psi}{\longrightarrow} A'' \longrightarrow 0 $,
	one can find a refinement $ \mu:\mathcal{H}\rightarrow\mathcal{G}  $ such that
	$ \mu^*(b)\in C^{n-1}(X,\mathcal{H};Im_{A''}) $. Therefore,
	\begin{center}
		$ (\lambda\circ\mu)^*(c)=\mu^*\circ\lambda^*(c) =\mu^*\circ\delta^{n-1}(b) =\delta^{n-1}\circ\mu^*(b)=\delta^{n-1}\circ\psi_{*}(a),  $
	\end{center}
	and so $ [c]=[(\lambda\circ\mu)^*(c)]=0 $ in $ \check{H}^n(X;Im_{A''}) $.
	
	Now by replacing $ \check{H}^n(X;A'') $ with $ \check{H}^n(X;Im_{A''}) $ under the isomorphism $ \theta^n_{\#} $ in (\ref{eq1}) and in view of Proposition \ref{p-0th}, the result is achieved.
\end{proof}

\section{The de Rham theorem}\label{S5}   
The aim of this section is to provide a setup that gives rise to the generalized Mayer-Vietoris sequence and the de Rham theorem. 
\subsection{The generalized Mayer-Vietoris sequence}

\begin{definition}\label{def-pu} 
	Let $ X $ be a diffeological space.
	A \textbf{smooth partition of unity subordinate to a covering generating family} $ \mathcal{C}=\{P_{\alpha}\}_{\alpha\in I} $ of $ X $ is 
	\begin{enumerate}
		\item[(a)] a collection $ \{ f_{\alpha}:X\rightarrow[0,1]\}_{\alpha\in I} $ of smooth functions on $ X $, where $ [0,1]\subseteq \mathbb{R} $ is equipped with the subspace diffeology, 
		
		\item[(b)] an assignment to each plot $ Q  $ in $ X $, 
		a  collection $ \{Q_{\alpha}  \}_{\alpha\in I} $ of  plots 	with $ Q=\mathsf{sup}_{\alpha\in I}Q_{\alpha} $,
	\end{enumerate}
	such that
	$ \{ Q^*(f_{\alpha}) \}_{\alpha\in I} $ is a  smooth partition of unity subordinate to $ \{ \mathrm{dom}(Q_{\alpha})\}_{\alpha\in I} $, i.e.,
	\begin{enumerate}
		\item[(i)]
		$ {\rm{supp}}\big{(}Q^*(f_{\alpha})\big{)}\subseteq \mathrm{dom}(Q_{\alpha}) $ for each $ \alpha\in I $,
		\item[(ii)]
		$ \{{\rm{supp}}\big{(}Q^*(f_{\alpha})\big{)}\}_{\alpha\in I} $ is locally finite, 
		\item[(iii)]
		$ \displaystyle\sum_{\alpha\in I} Q^*(f_{\alpha})(r)=1 $ for all $ r\in\mathrm{dom}(Q) $. 
	\end{enumerate}
\end{definition}

\begin{example} 
	Let $ M $ be a usual manifold with an atlas $ \mathfrak{A}=\{ \varphi_{\alpha}:U_{\alpha}\rightarrow M\}_{\alpha\in I} $ (as a covering generating family).
	Then the pullback of any smooth partition of unity  $ \{ f_{\alpha}:X\rightarrow[0,1]\}_{\alpha\in I} $ subordinate to $ \{ \mathrm{Im}(\varphi_{\alpha}) \}_{\alpha\in I} $ induces a smooth partition of unity subordinate to $ \mathfrak{A} $.
\end{example}

\begin{setup} \label{setting-1}
	Let $ \mathcal{C}=\{P_{\alpha}\}_{\alpha\in I} $ be a covering generating family for a diffeological space $ (X,\mathcal{D}_X) $, equipped with the diffeology $ \mathcal{D}_X $. For every $ \alpha\in I $ and each plot $ Q $ in $ X $, suppose we are given
	a collection
	\begin{center}
		$ \Delta_0(\alpha,Q)= \{Q^{\alpha\beta }\stackrel{G_Q^{\alpha\beta }}{\longrightarrow} P_{\alpha} \}_{\beta\in J_{\alpha}} $
	\end{center}
	of morphisms of plots from some restrictions $ Q^{\alpha\beta } $ of $ Q $ into $ P_{\alpha} $,
	such that the sources of elements of
	\begin{center}
		$ \Delta_0(Q)=\bigcup_{ \alpha\in I }\Delta_0(\alpha,Q) $ 
	\end{center}
	constitute a covering for $ Q $, i.e.,
	the collection $ \{ \mathrm{dom}(Q^{\alpha\beta})\}_{\alpha\in I,\beta\in J_{\alpha}} $ covers $ \mathrm{dom}(Q) $.
	Also, assume that morphisms of plots are \textbf{locally extendable} on 
	\begin{center}
		$ \mathrm{Mor}_{loc}(\mathcal{C})=\bigcup_{ Q\in\mathcal{D}_X}\Delta_0(Q) $, 
	\end{center}
	meaning that, for given morphisms $ Q\stackrel{G}{\longrightarrow} P  $ and $ Q\stackrel{F}{\longrightarrow} Q' $ of plots in $ X $ with $ Q\stackrel{G}{\longrightarrow} P  $ belonging to $ \mathrm{Mor}_{loc}(\mathcal{C}) $, and for every $ r\in\mathrm{dom}(Q) $, there are an open neighborhood $ V\subseteq\mathrm{dom}(Q) $ of $ r $  and an element
	$ Q' \stackrel{G'}{\longrightarrow} P $ of $ \mathrm{Mor}_{loc}(\mathcal{C})  $ such that the following diagram commutes:
	\begin{displaymath}
		\xymatrix{
			&P \\
			Q'\ar@{.>}[ur]^{G'} & Q|_{V}\ar[u]_{G|_{V}}\ar[l]^{F|_{V}} 	 	} 
	\end{displaymath}
	
\end{setup}

\begin{construction}
	\em
	In the situation of Setup \ref{setting-1},
	for each $ 0 $-simplex $ P_0\in \mathcal{C} $, 
	we automatically get a  collection 
	\begin{center}
		$ \Delta_1(P_0)=\{ P_0\stackrel{\imath_{\alpha\beta}}{\longleftarrow} P_0^{\alpha\beta}\stackrel{G_0^{\alpha\beta}}{\longrightarrow} P_{\alpha}\} $ 
	\end{center}
	of $ 1 $-simplexes on $ \mathcal{C} $ in which $ P_0^{\alpha\beta}\stackrel{G_0^{\alpha\beta}}{\longrightarrow} P_{\alpha} $ belongs to $ \Delta_0(P_0) $ and $ P_0^{\alpha\beta}\stackrel{\imath_{\alpha\beta}}{\longrightarrow} P_0 $ is the canonical inclusion. Notice that the collection of nerve plots of elements of $ \Delta_1(P_0) $ forms a covering for $ P_0$.
	
	Now let $ P_0\stackrel{F_1}{\longleftarrow}Q\stackrel{F_0}{\longrightarrow} P_1 $ be a $ 1 $-simplex on $ \mathcal{C} $ and $ r\in\mathrm{dom}(Q) $. Then $ r\in\mathrm{dom}(Q^{\alpha\beta}) $ for some $ Q^{\alpha\beta} $.
	Since morphisms of plots are locally extendable on $ \mathrm{Mor}_{loc}(-,\mathcal{C}) $, for the corresponding morphism $ Q^{\alpha\beta }\stackrel{G_Q^{\alpha\beta }}{\longrightarrow} P_{\alpha} \in \Delta_0(Q) $,
	we have morphisms $ P_0^{\alpha\beta_0}\stackrel{H_0^{\alpha\beta_0}}{\longrightarrow} P_{\alpha}\in\Delta_0(P_0) $ and $ P_0^{\alpha\beta}\stackrel{H_1^{\alpha\beta_1}}{\longrightarrow} P_{\alpha_1}\in\Delta_0(P_1) $ 
	and an open neighborhood $ V\subseteq\mathrm{dom}(Q) $ of $ r $
	such that 
	\begin{displaymath}
		\xymatrix{
			&P_{\alpha}& \\
			P_0^{\alpha\beta_0}\ar[ur]^{H_0^{\alpha\beta_0}} & Q^{\alpha\beta}|_{V}\ar[u]_{G_Q^{\alpha\beta }}\ar[l]^{F_0|_{V}}\ar[r]_{F_1|_{V}} &	 P_1^{\alpha\beta_1}\ar[ul]_{H_1^{\alpha\beta_1}}	} 
	\end{displaymath}
	This diagram can be completed via the canonical inclusions so that a   $ 2 $-simplexes on $ \mathcal{C} $ is obtained as the following: 
	
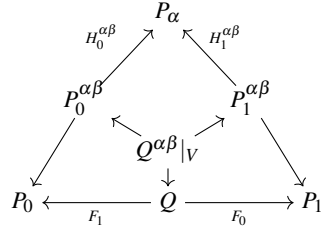
\begin{center}
		\begin{tikzpicture}
			\matrix (m) [matrix of math nodes, row sep=.8em,
			column sep=0.6em, text height=1.2ex, text depth=0.25ex]
			{ & & P_{\alpha} & & \\
				& &  & & \\
				& &  & & \\
				& P_0^{\alpha\beta}  & & P_1^{\alpha\beta} & \\
				& & Q^{\alpha\beta}|_{V} & & \\
				P_0 & & Q & & P_1 \\};
			\path[->,font=\scriptsize]
			(m-4-2) edge node[auto] {$H_0^{\alpha\beta} $} (m-1-3) edge node[above] {} (m-6-1)
			(m-4-4) edge node[above] {$\qquad H_1^{\alpha\beta} $} (m-1-3) edge node[auto] {} (m-6-5)
			(m-6-3) edge node[auto] {$ F_1 $} (m-6-1) edge node[below] {$ F_0 $} (m-6-5)
			(m-5-3) edge node[auto]  { } (m-4-2) edge node[below]  { } (m-4-4) edge node[auto]  { } (m-6-3);
		\end{tikzpicture} 
	\end{center}
	with the nerve plot $ Q^{\alpha\beta}|_{V} $
	in which $ P_0\stackrel{\imath_{\alpha\beta}}{\longleftarrow} P_0^{\alpha\beta}\stackrel{H_0^{\alpha\beta}}{\longrightarrow} P_{\alpha}\in\Delta_1(P_0) $ and $ P_1\stackrel{\imath_{\alpha\beta}}{\longleftarrow} P_0^{\alpha\beta}\stackrel{H_1^{\alpha\beta}}{\longrightarrow} P_{\alpha}\in\Delta_1(P_1) $.
	Hence, for each $ 1 $-simplex $ P_0\stackrel{F_1}{\longleftarrow}Q\stackrel{F_0}{\longrightarrow} P_1 $ on $ \mathcal{C} $,
	we have a collection $ \Delta_2(P_0,P_1)=\{ (P_{\alpha},P_0,P_1)\} $ of $ 2 $-simplexes on $ \mathcal{C} $ 
	such that the nerve plots of elements of $ \Delta_2(P_0,P_1) $ constitute a covering for $ Q$, and 
	$ (P_{\alpha},P_1)\in\Delta_1(P_1) $ and $ (P_{\alpha},P_0)\in\Delta_1(P_0) $.
	
	Inductively, in the situation of Setup \ref{setting-1},		
	for any $ n $-simplex $ (P_0,\ldots,P_n) $ on $ \mathcal{C} $ with the nerve plot $Q$, we can construct a collection $ \Delta_{n+1}(P_0,\ldots,P_{n})=\{ (P_{\alpha},P_0,\ldots,P_{n})\} $ of $ (n+1) $-simplexes on $ \mathcal{C} $ 
	with the following properties:
	\begin{enumerate}
		\item[$ \bullet $]
		The collection of nerve plots of elements of $ \Delta_{n+1}(P_0,\ldots,P_{n}) $ constitutes a covering for $ Q$. In addition, since domains are paracompact, in each step we can assume that the obtained covering for  $ Q $ is locally finite. 
		\item[$ \bullet $]
		The morphism $ Q_{(P_{\alpha},P_0,\ldots,P_{n})}\stackrel{\imath}{\longrightarrow} Q_{(P_0,\ldots,P_{n})} $ is the canonical inclusion, and 
		for each $ i=0,\ldots,n $, $ Q_{(P_{\alpha},P_0,\ldots,P_{n})}\stackrel{F_i\circ\imath }{\longrightarrow} Q_{(P_{\alpha},P_0,\ldots,\widehat{P_i},\ldots,P_{n})} $
		is a restriction of the morphism $ Q\stackrel{F_i}{\longrightarrow} Q_i $ given by the $ n $-simplex $ (P_0,\ldots,P_n) $.
		
		\item[$ \bullet $]
		$ (P_{\alpha},P_0,\ldots,\widehat{P_i},\ldots,P_{n}) $ belongs to $ \Delta_{n}(P_0,\ldots,\widehat{P_i},\ldots,P_{n}) $, for each $ i=1,\ldots,n $. 
	\end{enumerate}

	Now suppose that there exists a smooth partition of unity $ \{ f_{\alpha}:X\rightarrow[0,1]\}_{\alpha\in I} $ subordinate to $ \mathcal{C}=\{P_{\alpha}\}_{\alpha\in I} $ and consider the operator
	\begin{center}
		$ K^n:C^{n+1}(X,\mathcal{C},\Lambda^k)\longrightarrow C^n(X,\mathcal{C},\Lambda^k) $
	\end{center}
	given by
	\begin{center}
		$ (K^n\omega)(P_0,\ldots,P_{n})=\displaystyle\sum_{\Delta_{n+1}(P_0,\ldots,P_{n})} Q_{(P_{\alpha},P_0,\ldots,P_{n})}^*(f_{\alpha})~~~\omega(P_{\alpha},P_0,\ldots,P_{n}), $
	\end{center}
	where $ Q_{(P_{\alpha},P_0,\ldots,P_{n})} $ denotes the nerve plot of an $ (n+1) $-simplex $ (P_{\alpha},P_0,\ldots,P_{n}) $ constructed from $ n $-simplex $ (P_0,\ldots,P_{n}) $ with the nerve plot $ Q $, as above. Each term 
	\begin{center}
		$ Q_{(P_{\alpha},P_0,\ldots,P_{n})}^*(f_{\alpha})~~~\omega(P_{\alpha},P_0,\ldots,P_{n}) $
	\end{center}
	is defined on the whole $ \mathrm{dom}(Q) $ to be zero outside $ \mathrm{dom}(Q_{(P_{\alpha},P_0,\ldots,P_{n})})\setminus{\rm{supp}}(Q^*(f_{\alpha})) $.
\end{construction}
\begin{proposition} 
	$ K^n\circ\delta^n+\delta^{n-1}\circ K^{n-1}=\mathrm{id} $.
\end{proposition}
\begin{proof}
	For every $ \omega\in C^n(X,\mathcal{C},\Lambda^k) $ and $ n $-simplex $ (P_0,\ldots,P_n) $ with the nerve plot $ Q $, we have
	\begin{align*}
		K^n\circ\delta^n(\omega)&(P_0,\ldots,P_n)=\displaystyle\sum_{\Delta_{n+1}(P_0,\ldots,P_{n})} Q_{(P_{\alpha},P_0,\ldots,P_{n})}^*(f_{\alpha})~~~~~~(\delta^n\omega)(P_{\alpha},P_0,\ldots,P_{n})\\
		&= \displaystyle\sum_{\Delta_{n+1}(P_0,\ldots,P_{n})} Q_{(P_{\alpha},P_0,\ldots,P_{n})}^*(f_{\alpha})~~~\Big{(}\imath_{\alpha}^*\omega(P_0,\ldots,P_n)-\sum_{i=0}^{n}(-1)^i~(F_{i}\circ \imath_{\alpha})^*~ \omega (P_{\alpha},P_0,\ldots,\widehat{P_i},\ldots,P_{n})\Big{)}\\
		&= \displaystyle\sum_{\Delta_{n+1}(P_0,\ldots,P_{n})} Q_{(P_{\alpha},P_0,\ldots,P_{n})}^*(f_{\alpha}) \imath_{\alpha}^*\omega(P_0,\ldots,P_n)\\
		&-\sum_{\Delta_{n+1}(P_0,\ldots,P_{n})} Q_{(P_{\alpha},P_0,\ldots,P_{n})}^*(f_{\alpha})\sum_{i=0}^{n}(-1)^i~(F_{i}\circ \imath_{\alpha})^*~ \omega (P_{\alpha},P_0,\ldots,\widehat{P_i},\ldots,P_{n})\\
		&= \omega(P_0,\ldots,P_n)-\sum_{\Delta_{n+1}(P_0,\ldots,P_{n})}\sum_{i=0}^{n}(-1)^i~~ Q_{(P_{\alpha},P_0,\ldots,P_{n})}^*(f_{\alpha})~(F_{i}\circ \imath_{\alpha})^*~ \omega (P_{\alpha},P_0,\ldots,\widehat{P_i},\ldots,P_{n}).
	\end{align*}
	On the other hand, 
	\begin{align*}
		\delta^{n-1}\circ K^{n-1}&(\omega)(P_0,\ldots,P_n)=\displaystyle\sum_{i=0}^{n}(-1)^i~F_i^*~ \Big{(}(K^{n-1}\omega) (P_0,\ldots,\widehat{P_i},\ldots,P_{n})\Big{)}\\
		&=\displaystyle\sum_{i=0}^{n}(-1)^i~F_i^*~ \Big{(}\sum_{\Delta_{n}(P_0,\ldots,\widehat{P_i},\ldots,P_{n})} Q_{P_{\alpha},P_0,\ldots,\widehat{P_i},\ldots,P_{n}}^*(f_{\alpha})~~~\omega(P_{\alpha},P_0,\ldots,\widehat{P_i},\ldots,P_{n})\Big{)}\\
		&=\displaystyle\sum_{i=0}^{n} \sum_{\Delta_{n}(P_0,\ldots,\widehat{P_i},\ldots,P_{n})} (-1)^i~(F_{i}\circ \imath_{\alpha})^*\Big{(}Q_{P_{\alpha},P_0,\ldots,\widehat{P_i},\ldots,P_{n}}^*(f_{\alpha})\Big{)}~~~(F_{i}\circ \imath_{\alpha})^*\omega(P_{\alpha},P_0,\ldots,\widehat{P_i},\ldots,P_{n})\\
		&=\displaystyle\sum_{i=0}^{n}\sum_{\Delta_{n+1}(P_0,\ldots,P_{n})}(-1)^i~~ Q_{(P_{\alpha},P_0,\ldots,P_{n})}^*(f_{\alpha})~(F_{i}\circ \imath_{\alpha})^*~ \omega (P_{\alpha},P_0,\ldots,\widehat{P_i},\ldots,P_{n}).
	\end{align*}
	Therefore,
	\begin{center}
		$ 	K^n\circ\delta^n(\omega)(P_0,\ldots,P_n)+\delta^{n-1}\circ K^{n-1}(\omega)(P_0,\ldots,P_n)=\omega(P_0,\ldots,P_n) $
	\end{center}
\end{proof}
Now the following result is immediate:
\begin{proposition}\label{prop-mv}  
	(The generalized Mayer-Vietoris sequence).
	Let $ X $ be a diffeological space and let $ \mathcal{C} $ be a covering generating family for $ X $.
	In the situation of Setup \ref{setting-1}, if there exists a partition of unity subordinate to $\mathcal{C}$,
	then the sequence 
	\begin{align*}
		0\longrightarrow \Omega^{k}(X)&
		\stackrel{r}{\longrightarrow}C^0(X,\mathcal{C},\Lambda^k)
		\stackrel{\delta}{\longrightarrow}C^1(X,\mathcal{C},\Lambda^k)
		\stackrel{\delta}{\longrightarrow}C^2(X,\mathcal{C},\Lambda^k)
		\longrightarrow \cdots
	\end{align*}
	is exact, or equivalently, $ H^n(X,\mathcal{C};\Lambda^k)=0 $.
\end{proposition}

\subsection{Abstract de Rham theorem}
\begin{definition} 
	A sheaf  $ A $ on a diffeological space $ X $ is \textbf{acyclic} if $ \check{H}^n(X;A)=0 $ for all $ n\geq 1 $.
\end{definition}

\begin{theorem}(Abstract de Rham theorem).
	Assume that
	\begin{center}
		$0\longrightarrow \mathcal{S} \longrightarrow A^{0}\stackrel{d_0}{\longrightarrow}A^{1}\stackrel{d_1}{\longrightarrow}A^{2}\longrightarrow\cdots $.
	\end{center}
	is a resolution of a sheaf $ \mathcal{S} $ on a diffeological space $ X $ in which $ A^{k} $ is an acyclic sheaf for all $ k\geq 1 $.
	Then 
	$ \check{H}^n(X;\mathcal{S}) $ and $ H^n(\Sigma A^{\bullet}) $ are isomorphic, where $ H^n(\Sigma A^{\bullet}) $ denotes the $ k $-th cohomology group of the associated cochain complex
	\begin{center}
		$ \Sigma A^{0}\stackrel{\Sigma d_0}{\longrightarrow} \Sigma A^{1}\stackrel{\Sigma d_1}{\longrightarrow}\Sigma A^{2}\longrightarrow\cdots $.
	\end{center}
\end{theorem}
\begin{proof}
	Set $ Z^0=\mathcal{S} $, and for $ k> 0 $, let $ Z^k $ be the sheaf $ \mathrm{ker}(d_k) $ so that we have a short exact sequence
	\begin{center}
		$ 0\longrightarrow Z^k\stackrel{\imath}{\longrightarrow} A^k\stackrel{d_k}{\longrightarrow} Z^{k+1}\longrightarrow 0 $
	\end{center}
	of sheaves, for all $ k\geq 0 $.
	By proposition \ref{p-long}, this gives rise to a long exact sequence
	\begin{align}\label{eq2}
		0\longrightarrow&\Sigma Z^k(X)\stackrel{\Sigma\imath}{\longrightarrow}\Sigma A^k(X)
		\stackrel{\Sigma d_k}{\longrightarrow}\Sigma Z^{k+1}(X)\stackrel{}{\longrightarrow}\check{H}^{1}(X;Z^k)
		\stackrel{}{\longrightarrow}\check{H}^1(X;A^k)
		\\
		\cdots	\longrightarrow& \check{H}^n(X;Z^k)
		\stackrel{}{\longrightarrow}\check{H}^n(X;A^k)
		\stackrel{}{\longrightarrow}\check{H}^n(X;Z^{k+1})
		\stackrel{\partial}{\longrightarrow}\check{H}^{n+1}(X;Z^k)
		\longrightarrow \cdots
	\end{align}
	Since  $ A^{k} $ is an acyclic sheaf for $ k\geq 1 $, 
	$ \partial:\check{H}^n(X;Z^{k+1})\longrightarrow \check{H}^{n+1}(X;Z^k) $
	are isomorphisms, for $ n\geq 1 $. Thus, inductively, the isomorphisms
	\begin{center}
		$ \check{H}^n(X;\mathcal{S})=\check{H}^n(X;Z^{0})\cong \check{H}^{1}(X;Z^{n-1}) $
	\end{center}
	are achieved. In view of (\ref{eq2}), we have an exact sequence
	\begin{center}
		$ 0\longrightarrow\Sigma Z^{n-1}(X)\stackrel{\Sigma\imath}{\longrightarrow}\Sigma A^{n-1}(X)
		\stackrel{\Sigma d_{n-1}}{\longrightarrow}\Sigma Z^{n}(X)\stackrel{}{\longrightarrow}\check{H}^{1}(X;Z^{n-1})
		\stackrel{}{\longrightarrow}0 $
	\end{center}
	which yields the isomorphisms
	\begin{center}
		$ \check{H}^n(X;\mathcal{S})\cong \check{H}^{1}(X;Z^{n-1})\cong \Sigma Z^{n}(X)/Im(\Sigma d_{n-1})\cong H^n(\Sigma A^{\bullet}) $
	\end{center}
	for $ n\geq 1 $.
\end{proof}
In view of Propositions \ref{prop-mv} and \ref{prop-res} and the abstract de Rham theorem, the following is immediate.

\begin{corollary}
	Let $ X $ be a diffeological space. In the situation of Setup \ref{setting-1},
	if for every covering generating family $\mathcal{C}$ of $ X $, there exists a partition of unity subordinate to $\mathcal{C}$,
	then the sheaf $ \Lambda^k $, $ k\geq 1 $,  of differential $ k $-forms on $ X $ is acyclic. In particular, we have
	$ \check H^n(X;\mathbb{R})\cong H_{dR}^n(X) $.
\end{corollary}

\subsection{The case of diffeological \'{e}tale manifolds}
Let us first review the definition of diffeological \'{e}tale manifolds (see \cite{ARA} for details).
\begin{definition}
	A map $ f:X\rightarrow Y $ between diffeological spaces is \textbf{\'{e}tale} if for every $ x $ in $ X $, there are D-open neighborhoods $ O\subseteq X $ and $ V\subseteq Y $ of $ x $ and $ f(x) $, respectively, such that $ f|_O:O\rightarrow O' $ is a diffeomorphism.  
	A smooth map $ f:X\rightarrow Y $ is a \textbf{diffeological \'{e}tale map} if the pullback $ P^{*}f $ by every plot $ P $ in $ X $ is  \'{e}tale. 
\end{definition}
Any diffeological \'{e}tale map  is a D-open map. 
\begin{proposition}
	A smooth map $ f:X\rightarrow Y $ is a diffeological \'etale map if and only if for given $ x\in X $, any plot $ P:U\rightarrow Y $, and each $ r\in U $ with $ P(r)=f(x)  $, there exists a unique (up to germs of plots)  local lift plot $ L:V\rightarrow X $ defined on an open neighborhood $ V\subseteq U $ of $ r $ such that $ f\circ L=P|_V $ and $ L(r)=x $.
\end{proposition}

\begin{definition} 
	A diffeological space $ \mathcal{M} $ is said to be a \textbf{diffeological \'{e}tale $ n $-manifold} if there exists a parametrized cover $ \mathfrak{A} $, called an \textbf{atlas}, for $ \mathcal{M} $ consisting of diffeological \'{e}tale maps 
	from $ n $-domains into $ \mathcal{M} $.
	We call the elements of $ \mathfrak{A} $ \textbf{diffeological \'{e}tale charts}.
	In this situation, $ \mathfrak{A} $ is actually a covering generating family for $ \mathcal{M} $.
\end{definition}

If $ \varphi:U\rightarrow \mathcal{M} $ are $ \psi:V\rightarrow \mathcal{M} $ are two diffeological \'{e}tale charts in $ \mathcal{M} $ with $ \varphi(r)=\psi(s) $, then one can find a unique smooth map $ h:U'\rightarrow V $ defined on an open neighborhood $ U'\subseteq U $ of $ r $ such that $ h(r)=s $ and the following diagram commutes:
\begin{displaymath}
	\xymatrix{
		&  \mathcal{M}  & \\
		U' \ar[ur]^{\varphi|_{U'}}\ar[rr]_{h} & & V\ar[ul]_{\psi}  \\
	}
\end{displaymath}
$ h $ is an \'{e}tale map and uniquely determined by diffeological \'{e}tale charts $ \varphi $ are $ \psi $. 

\begin{example}\label{exa-1}
	Any diffeological manifold is a diffeological \'{e}tale manifold.
	The irrational torus $ \mathbb{T}_{\alpha}=\mathbb{R}/(\mathbb{Z}+\alpha\mathbb{Z}) $,  $ \alpha\notin\mathbb{Q}  $,  is a  diffeological \'{e}tale $ 1 $-manifold.
\end{example}

Assume that $ \mathcal{M} $ is a diffeological \'{e}tale manifold along with an atlas $ \mathfrak{A} $.
First, we observe how  $ \mathcal{M} $ along with an atlas $ \mathfrak{A}=\{\varphi_{\alpha}:U_{\alpha}\rightarrow\mathcal{M}\}_{\alpha\in I} $ fulfills the requirements of Setup \ref{setting-1}: 

For each  $ \alpha\in I $, choose a selection\footnote{Thanks to the axiom of choice.} $ \tau_{\alpha}:\mathrm{Im}(\varphi_{\alpha})\rightarrow U_{\alpha}  $ of $ \varphi_{\alpha} $ so that
$ \varphi_{\alpha}\circ\tau_{\alpha}=\mathrm{id} $.
Let $ \alpha\in I $, $ Q $ be an arbitrary plot in $ \mathcal{M} $, and $ r\in\mathrm{dom}(Q) $. 
If $ Q(r)\notin\mathrm{Im}(\varphi_{\alpha}) $, take the morphism $ \varnothing\rightarrow\varphi_{\alpha} $ from the empty plot.
But if $ Q(r)\in\mathrm{Im}(\varphi_{\alpha}) $, then there exists a smooth map $ G_{\alpha r}:V_{\alpha r}\rightarrow U_{\alpha} $ defined on an open neighborhood $ V_{\alpha r}\subseteq\mathrm{dom}(Q) $ of $ r $
such that $ G_{\alpha r}(r)=\tau_{\alpha}(Q(r)) $, which is  unique with this property (up to germs). This gives us a morphism $ Q_{\alpha r}\stackrel{G_{\alpha r}}{\longrightarrow} \varphi_{\alpha} $, where $ Q_{\alpha r} $ is  the restriction of $ Q $ to $ V_{\alpha r} $. Let $ \Delta_0(\alpha,Q) $ be the collection of  morphisms chosen in this manner.
Let 	 $ \Delta_0(Q)=\bigcup_{ \alpha\in I }\Delta_0(\alpha,Q) $. Since $ \mathfrak{A} $ is a parametrized cover of $ \mathcal{M} $, the collection
$ \{\mathrm{dom}(Q_{\alpha r})\}_{\alpha\in I, r\in\mathrm{dom}(Q) } $ covers $ \mathrm{dom}(Q) $.

To observe that morphisms of plots are  locally extendable  on $ \mathrm{Mor}_{loc}(\mathfrak{A})=\bigcup_{Q\in\mathcal{D}_{\mathcal{M}}}\Delta_0(Q) $, suppose we are given morphisms $ Q\stackrel{G}{\longrightarrow} \varphi_{\alpha}  $ and $ Q\stackrel{F}{\longrightarrow} Q' $ of plots in $ \mathcal{M} $ with $ Q\stackrel{G}{\longrightarrow} \varphi_{\alpha}  $ belonging to $ \Delta_0(Q) $. Let $ r\in\mathrm{dom}(Q) $ so that $ G(r)=\tau_{\alpha}(Q(r)) $.
Then for $ F(r)\in\mathrm{dom}(Q') $, take the element $ Q'_{\alpha F(r)}\stackrel{G'_{\alpha F(r)}}{\longrightarrow} \varphi_{\alpha} $ of $ \Delta_0(Q') $  for which $ G'_{\alpha F(r)}(F(r))=\tau_{\alpha}(Q'(F(r)))=\tau_{\alpha}(Q(r)) $.
By uniqueness, we get $ G'_{\alpha F(r)}\circ F|_W=G|_W $ on some open neighborhood $ W\subseteq\mathrm{dom}(Q) $ of $ r $.

Now in view of Proposition \ref{prop-mv}, we obtain the following results:
\begin{proposition}
	Assume that $ \mathcal{M} $ is a diffeological \'{e}tale manifold along with an atlas $ \mathfrak{A} $.
	If there exists a partition of unity subordinate to $\mathfrak{A}$,
	then the sequence 
	\begin{align*}
		0\longrightarrow \Omega^{k}(\mathcal{M})&
		\stackrel{r}{\longrightarrow}C^0(\mathcal{M},\mathfrak{A},\Lambda^k)
		\stackrel{\delta}{\longrightarrow}C^1(\mathcal{M},\mathfrak{A},\Lambda^k)
		\stackrel{\delta}{\longrightarrow}C^2(\mathcal{M},\mathfrak{A},\Lambda^k)
		\longrightarrow \cdots
	\end{align*}
	is exact, or equivalently, $ H^n(\mathcal{M},\mathfrak{A};\Lambda^k)=0 $.
\end{proposition}

\begin{proposition}
	Let $ (\mathcal{M},\mathfrak{A}) $ be a diffeological \'{e}tale manifold.
	If for every atlas $\mathfrak{A}'$ of $ X $, there exists a partition of unity subordinate to $\mathfrak{A}'$,
	then 
	$ \check H^n(\mathcal{M};\mathbb{R})\cong H_{dR}^n(\mathcal{M}) $.
\end{proposition}

\begin{proof}
	We first notice that for any two covering generating families $ \mathcal{C} $ and $ \mathcal{C}' $ of $ \mathcal{M} $, we can construct an atlas $\mathfrak{B}$ of $ \mathcal{M} $, consisting of restrictions of elements of the atlas $ \mathfrak{A} $, which refines both $ \mathcal{C} $ and $ \mathcal{C}' $. 
	Therefore, the sheaf $ \Lambda^k $, $ k\geq 1 $,  of differential $ k $-forms on $ X $ is acyclic, and 	$ \check H^n(\mathcal{M};\mathbb{R})\cong H_{dR}^n(\mathcal{M}) $.
\end{proof}

\begin{corollary} 
	If $ M $ is a usual manifold,
	then
	$ \check H^n(M;\mathbb{R})\cong H_{dR}^n(M) $.
\end{corollary}
\begin{proof}
	Since there always exists a partition of unity subordinate to each atlas of a usual manifold.
\end{proof}
\subsection{Diffeological \v{C}ech-de Rham double complex} 
Let $ X $ be a diffeological space
and let $\mathcal{C}$ be the covering generating family consisting of global plots in $ X $. 
Consider the double complex 
$ C^{\bullet,\bullet}(X) $ with
\begin{center}
	$ C^{n,k}(X)= C^n(X,\mathcal{C},\Lambda^k) $
\end{center}
in which $ \Lambda^k $ is the sheaf of differential $k$-forms, together with differentials 
\begin{center}
	$ \delta:C^{n,k}(X)\rightarrow C^{n+1,k}(X)\qquad$ 
	and 
	$\qquad d:C^{n,k}(X)\rightarrow C^{n,k+1}(X) $,
\end{center} 
where
$ d(c) (P_0,\ldots,P_n)=(-1)^n d_Q (c(P_0,\ldots,P_n))$ and $ d_Q $ is the exterior derivative defined on $ \mathrm{dom}(Q) $, for any $ n $-cochain $ c $ and $ n $-simplex $ (P_0,\ldots,P_n) $ with the nerve  plot $Q$.
So we obtain the total complex $ C^{\bullet}(X) $ with 
\begin{center}
	$ C^m(X)=\displaystyle\bigoplus_{n+k=m}C^{n,k}(X) $
\end{center}
and differential 
$ D=\delta+d $.
\begin{displaymath}
	\xymatrix{
		\vdots & \vdots & \vdots &  \\
		C^0(X,\mathcal{C},\Lambda^2)\ar[r]^{\delta} \ar[u] & C^1(X,\mathcal{C},\Lambda^2)\ar[r]^{\delta} \ar[u] & C^2(X,\mathcal{C},\Lambda^2)\ar[r] \ar[u] & \cdots\\
		C^0(X,\mathcal{C},\Lambda^1)\ar[r]^{\delta} \ar[u]^{d} & C^1(X,\mathcal{C},\Lambda^1)\ar[r]^{\delta} \ar[u]^{d} & C^2(X,\mathcal{C},\Lambda^1)\ar[r] \ar[u]^{d} & \cdots\\
		C^0(X,\mathcal{C},\Lambda^0)\ar[r]^{\delta} \ar[u]^{d} & C^1(X,\mathcal{C},\Lambda^0)\ar[r]^{\delta}\ar[u]^{d} & C^2(X,\mathcal{C},\Lambda^0)\ar[r]\ar[u]^{d} & \cdots\\
	}
\end{displaymath}
This gives rise to two spectral sequences $\lbrace^{d}E_r\rbrace $ and $\lbrace^{\delta}E_r\rbrace $ with 
$ E^{n,k}_r \Longrightarrow H^{n+k}(C^{\bullet}(X))$. 
In the zeroth page, we have
\begin{align*}
	^{d}E_0^{n,k} &\simeq C^{n,k}(X),~~~~~~~~d_0=d\\
	^{\delta}E_0^{n,k} &\simeq C^{n,k}(X),~~~~~~~~d_0=\delta
\end{align*}
By taking cohomology, in the first page we obtain
\begin{align*}
	^dE_1^{n,k} &= C^n(X,\mathcal{C},H_{dR}^k),\\
	^{\delta}E_1^{n,k} &=H^n(X,\mathcal{C};\Lambda^k),
\end{align*}
where $ H_{dR}^k $ is the presheaf assigning to each plot $P:U\rightarrow X$ the group $ H_{dR}^k(U) $. In particular,
$ ^{d}E_1^{n,0}=C^n(X,\mathcal{C},\mathbb{R}) $
and
$ ^{\delta}E_1^{0,k}=\Omega^k(X) $. 
In the second page, we have
\begin{align*}
	^dE_2^{n,k} &= H^n(X,\mathcal{C};H_{dR}^k),\\
	^{\delta}E_2^{n,k} &=H^k(H^n(X,\mathcal{C};\Lambda^{\bullet})).
\end{align*}
In particular, 
$ ^{d}E_2^{n,0}=H^n(X,\mathcal{C},\mathbb{R}) $
and
$ ^{\delta}E_2^{0,k} =H^k(\Omega^{\bullet}(X))=H_{dR}^k(X) $ for every integer $ k $.
There is the edge homomorphism 
\begin{center}
	$ H_{dR}^k(X)\longrightarrow H^{k}(C^{\bullet}(X))$
\end{center}
induced by the chain map $ \Omega^{k}(X)\hookrightarrow C^0(X,\mathcal{C},\Lambda^k)=C^{0,k}(X)\hookrightarrow C^k(X)$.
In addition, the exact sequence of terms of low degree is as below:
\begin{center}
	$ 0\longrightarrow H_{dR}^1(X)\longrightarrow H^1(C^{\bullet}(X)) \longrightarrow   {}^{\delta}E_2^{1,0}\longrightarrow H_{dR}^2(X)\longrightarrow H^2(C^{\bullet}(X)) $.
\end{center}



\section{Diffeological fiber bundles with gauge group}\label{S6}

Throughout this section, assume that $ X $ and $ T $ are diffeological spaces and let 
$ \varrho:G\rightarrow \mathrm{Diff}(T) $ be a smooth faithful action of a (not necessarily abelian) diffeological group $ G $ on $ T $.
\begin{definition}\label{def-fbgt} 
	A \textbf{diffeological fiber bundle with gauge group $ G $ 
		and typical fiber $ T $} on $ X $ is a smooth map  $ \pi:E\rightarrow X $ such that there is a
	covering generating family $ \mathcal{C} $ of $ X $
	along with the following data:
	\begin{enumerate}
		\item[(a)]
		a pullback 
		\begin{displaymath}
			\xymatrix{
				\mathrm{dom}(P)\times T \ar[r]^{\qquad\psi_P}\ar[d]_{\Pr_1 } & E \ar[d]^{\pi} \\
				\mathrm{dom}(P) \ar[r]^{\quad P} & X  }
		\end{displaymath} 
		in $ \textsf{Diff} $, called a
		\textbf{trivialization}, for every $ P\in \mathcal{C} $, 
		\item[(b)]
		a smooth map
		$ g_{P_0\stackrel{F_1}{\leftarrow}Q\stackrel{F_0}{\rightarrow} P_1}:\mathrm{dom}(Q) \rightarrow G $, 
		called a \textbf{transition},
		for each $ 1 $-simplex $ P_0\stackrel{F_1}{\longleftarrow}Q\stackrel{F_0}{\longrightarrow} P_1 $ on $ \mathcal{C} $
		such that
		\begin{center}
			$ \psi_{P_1}\big{(}F_0(r),t\big{)}=\psi_{P_0}\big{(}F_1(r),\varrho g_{P_0\stackrel{F_1}{\leftarrow}Q\stackrel{F_0}{\rightarrow} P_1}(r) (t)\big{)}, $
		\end{center}
		for all $ (r,t)\in \mathrm{dom}(Q)\times T $.
	\end{enumerate}
	
\end{definition}
Notice that (a) implies that  $ \pi:E\rightarrow X $ is  a surjective diffeological submersion.
\begin{remark}\label{rem-1} 
	Since the commutative square in (a) is a pullback, there is a diffeomorphism
	$ \overline{\psi_P}:\mathrm{dom}(P)\times T\rightarrow P^*E $
	such that the following diagram commutes:
	\begin{displaymath}
		\xymatrix{
			\mathrm{dom}(P)\times T\ar@/^{-0.4cm}/[rdd]_{\Pr_1}\ar[rd]^{\overline{\psi_P}}\ar@/^0.4cm/[rrd]^{\psi_P}	&	 &  \\
			&	P^*E  \ar[r]^{P_{\#}}\ar[d]_{P^*\pi } & E \ar[d]^{\pi} \\
			&	\mathrm{dom}(P) \ar[r]^{P} & X  }
	\end{displaymath} 
	Consequently, for a fixed $ r\in \mathrm{dom}(P) $, the map $ \psi_{P,r}:T\rightarrow E_{P(r)},\quad t\mapsto \psi_{P}(r,t) $ is a diffeomorphism with the smooth inverse $ (\psi_{P,r})^{-1}:E_{P(r)}\rightarrow T,\quad \xi\mapsto \Pr_2\circ(\overline{\psi_P})^{-1}(r,\xi) $, where $ E_{P(r)}=\pi^{-1}(P(r))\subseteq E $ is endowed with the subspace diffeology inherited from $ E $.
	Hence for each $ 1 $-simplex $ P_0\stackrel{F_1}{\longleftarrow}Q\stackrel{F_0}{\longrightarrow} P_1 $ on $ \mathcal{C} $
	and  for all $ r\in \mathrm{dom}(Q)$ and $ t\in T $,
	we get
	\begin{center}
		$ (\psi_{P_0,F_1(r)})^{-1}\circ\psi_{P_1,F_0(r)}=\varrho g_{P_0\stackrel{F_1}{\leftarrow}Q\stackrel{F_0}{\rightarrow} P_1}(r). $
	\end{center}
\end{remark}	

\begin{definition} 
	Let $ \pi:E\rightarrow X $ and $ \pi':E'\rightarrow X $ be two diffeological fiber bundles with gauge group $ G $ 
	and fiber $ T $ on $ X $ corresponding to a smooth action $ \varrho:G\rightarrow \mathrm{Diff}(T) $ of $ G $ on $ T $ with the data given as Definition \ref{def-fbgt}. 
	An \textbf{isomorphism} $ \Phi $ between $ \pi $ and $ \pi' $ is an $ X $-equivalence of smooth projections along with
	a common refinement $ \mathcal{B}=\{R_\beta\}_{\beta\in J} $ of $ \mathcal{C}=\{P_\alpha\}_{\alpha\in I} $ and $ \mathcal{C}'=\{P'_{\alpha'}\}_{\alpha'\in I'} $,  
	also refinement maps $ \lambda:J\rightarrow I $ and $ \lambda':J\rightarrow I' $, a family
	$ \{P_{\lambda(\beta)}\stackrel{f_\beta}{\longleftarrow}R_\beta \stackrel{f'_\beta}{\longrightarrow} P'_{\lambda'(\beta)}\}_{\beta\in J} $ 
	of morphisms of plots,
	such that 
	\begin{enumerate}
		\item[$ \bullet $]
		for each $ \beta\in J $ there exists a smooth map $ h_{\beta}:\mathrm{dom}(R_\beta)\rightarrow G $ such that
		\begin{center}
			$ \Phi\circ\psi_{P_{\lambda(\beta)}}\big{(}f_\beta(r),t\big{)}=\psi'_{P'_{\lambda'(\beta)}}\Big{(}f'_\beta(r),\varrho h_{\beta}(r) (t)\Big{)}, $
		\end{center}	
		for all $ (r,t)\in \mathrm{dom}(R_\beta)\times T $. 
	\end{enumerate}
	This condition implies that
	\begin{center}
		$ \big{(}\psi'_{P'_{\lambda'(\beta)},f'_\beta(r)}\big{)}^{-1}\circ\Phi\circ\psi_{P_{\lambda(\beta)},f_\beta(r)}=\varrho h_{\beta}(r). $
	\end{center}	
	To be isomorphic defines an equivalence relation on the set of diffeological fiber bundles with gauge group $ G $ 
	and fiber $ T $ on $ X $.
	Denote by $ \mathrm{Bund}_T(X,G) $ the set of classes this relation.
\end{definition}

\subsection{Relationship with (non-abelian) diffeological \v{C}ech cohomology in degree 1} 

\begin{definition} 
	Let  $ \mathcal{C} $ be a covering generating family of $ X $.
	A \textbf{diffeological \v{C}ech 1-cocycle} with values in $ G $ is an assignment to each 1-simplex  
	$ P_0\stackrel{F_1}{\leftarrow}Q\stackrel{F_0}{\rightarrow} P_1 $ on $ \mathcal{C} $
	a smooth map
	$ g_{P_0\stackrel{F_1}{\leftarrow}Q\stackrel{F_0}{\rightarrow} P_1}:\mathrm{dom}(Q) \rightarrow G $
	such that for any $2$-simplex on $ \mathcal{C} $
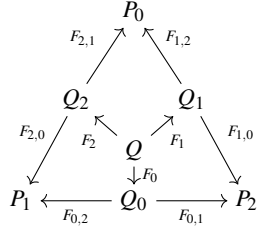
	\begin{center}
		\begin{tikzpicture}
			\matrix (m) [matrix of math nodes, row sep=.8em,
			column sep=0.6em, text height=1.2ex, text depth=0.25ex]
			{ & & P_0 & & \\
				& &  & & \\
				& &  & & \\
				& Q_2  & & Q_1 & \\
				& & Q & & \\
				P_1 & & Q_0 & & P_2 \\};
			\path[->,font=\scriptsize]
			(m-4-2) edge node[auto] {$ F_{2,1} $} (m-1-3) edge node[above] {$F_{2,0}\qquad$} (m-6-1)
			(m-4-4) edge node[above] {$\qquad F_{1,2} $} (m-1-3) edge node[auto] {$F_{1,0} $} (m-6-5)
			(m-6-3) edge node[auto] {$ F_{0,2} $} (m-6-1) edge node[below] {$F_{0,1} $} (m-6-5)
			(m-5-3) edge node[auto]  {$ F_2 $} (m-4-2) edge node[below]  {$\qquad F_1 $} (m-4-4) edge node[auto]  {$ F_0 $} (m-6-3);
		\end{tikzpicture} 
	\end{center}
	we have the cocycle indentity
	\begin{center}
		$ g_{P_0\stackrel{F_{2,1}}{\leftarrow}Q_2\stackrel{F_{2,0}}{\rightarrow} P_1}(F_2(r))\quad  g_{P_1\stackrel{F_{0,2}}{\leftarrow}Q_0\stackrel{F_{0,1}}{\rightarrow} P_2}(F_0(r))=g_{P_0\stackrel{F_{1,2}}{\leftarrow}Q_1\stackrel{F_{1,0}}{\rightarrow} P_2}(F_1(r)) $	
	\end{center}
	for all $ r\in \mathrm{dom}(Q) $.
	Two diffeological \v{C}ech 1-cocycles 
	\begin{center}
		$ \{g_{P_0\stackrel{F_1}{\leftarrow}Q\stackrel{F_0}{\rightarrow} P_1}\}_{1-\mathrm{simplexes}(\mathcal{C})} \qquad$ and $ \qquad\{g'_{P_0\stackrel{F_1}{\leftarrow}Q\stackrel{F_0}{\rightarrow} P_1}\}_{1-\mathrm{simplexes}(\mathcal{C})} $ 
	\end{center}
	are \textbf{equivalent} if there exists a family $ \{h_{P}:\mathrm{dom}(P)\rightarrow G\}_{P\in \mathcal{C}} $  of smooth maps such that
	\begin{center}
		$ g'_{P_0\stackrel{F_1}{\leftarrow}Q\stackrel{F_0}{\rightarrow} P_1}(r)\quad(h_{P_1}\circ F_0(r))=(h_{P_0}\circ F_1(r))\quad g_{P_0\stackrel{F_1}{\leftarrow}Q\stackrel{F_0}{\rightarrow} P_1}(r) $.
	\end{center}
	for each $ 1 $-simplex $ P_0\stackrel{F_1}{\longleftarrow}Q\stackrel{F_0}{\longrightarrow} P_1 $ on $ \mathcal{C} $.
	We denote the set of equivalence classes of diffeological \v{C}ech 1-cocycles by $ H^1(X,\mathcal{C},G) $.
\end{definition}
A refinement $ \lambda:\mathcal{C}'\rightarrow\mathcal{C}  $ induces a natural map 
$ \lambda^*:H^1(X,\mathcal{C};G)\rightarrow H^1(X,\mathcal{C}';G) $ taking $ \mathsf{class}\Big{(}\{g_{P_0\stackrel{F_1}{\leftarrow}Q\stackrel{F_0}{\rightarrow} P_1}\}_{1-\mathrm{simplexes}(\mathcal{C})}\Big{)} $ to $ \mathsf{class}\Big{(}\{g_{P_0\stackrel{F_1}{\leftarrow}Q\stackrel{F_0}{\rightarrow} P_1}\}_{1-\mathrm{simplexes}(\mathcal{C}')}\Big{)} $ by pullback.
Similar to the abelian case, this provides a map
$ j(\mathcal{C}',\mathcal{C}):H^1(X,\mathcal{C};A)\rightarrow H^1(X,\mathcal{C}';A) $
independent of the particular refinement map.
Therefore, we obtain a functor 
\begin{center}
	$ H^1(X,-;G):\mathsf{CGF}(X)^{op}\rightarrow\mathsf{Set} $.
\end{center}
\begin{definition} 
	We define  \textbf{the first non-abelian diffeological \v{C}ech cohomology} of $ X $ with coefficients in $ G $ as
	\begin{center}
		$ \check{H}^1(X;G)=\varinjlim_{\mathcal{C}}H^1(X,\mathcal{C};G). $
	\end{center}
\end{definition}

\begin{construction}\label{cons-1}
	Let we are given a diffeological \v{C}ech 1-cocycle $ \{g_{P_0\stackrel{F_1}{\leftarrow}Q\stackrel{F_0}{\rightarrow} P_1}\}_{1-\mathrm{simplexes}(\mathcal{C})} $ on a covering generating family $\mathcal{C}$. Consider the diffeological sum space
	\begin{center}
		$ \displaystyle\bigsqcup_{P\in\mathcal{C}}\mathrm{dom}(P)\times T $,
	\end{center}
	and define the equivalence relation $ \sim $ on it  by
	\begin{enumerate}
		\item[$ \blacklozenge $]
		$ (P,r,t)\sim(P',r',t') $ if and only if 
		there are morphisms $ P\stackrel{F}{\longleftarrow}Q\stackrel{F'}{\longrightarrow} P' $ such that $ F(s)=r$ and $ F'(s)=r' $ for some $ s\in \mathrm{dom}(Q)$, and $ t=\varrho g_{P\stackrel{F}{\leftarrow}Q\stackrel{F'}{\rightarrow} P'}(s)(t') $.
	\end{enumerate}
	We get the commutative diagram
	\begin{displaymath}
		\xymatrix{
			\bigsqcup_{P\in\mathcal{C}}\mathrm{dom}(P)\times T \ar[dr]_{\Pr} \ar[rrr]^{\qquad qt} & & & E \ar[dll]^{\pi}\\
			& X & &
		}
	\end{displaymath}
	of smooth maps,	where $ E $ is  the diffeological quotient space, $ q $ is the quotient map, and $ \Pr(P,r,t)=P(r) $ for every $ (P,r,t)\in\bigsqcup_{P\in\mathcal{C}}\mathrm{dom}(P)\times T $.
\end{construction}
\begin{remark}
	Notice that the condition $ \blacklozenge $ is equivalent to say that
	\begin{enumerate}
		\item[$  $]
		$ (P,r,t)\sim(P',r',t') $ if and only if $ P(r)=x=P'(r')  $
		and $ t=\varrho g_{P\stackrel{\textbf{r}}{\longleftarrow}\textbf{x}\stackrel{\textbf{r}'}{\longrightarrow} P'}(0)(t') $.
	\end{enumerate}
\end{remark}

\begin{proposition}\label{prop-cons}
	The smooth map 
	$ \pi:E\rightarrow X $ given by Construction \ref{cons-1} is a diffeological fiber bundle with gauge group $ G $ 
	and typical fiber $ T $ on $ X $ in which transitions are the given diffeological \v{C}ech 1-cocycles.
\end{proposition}
\begin{proof}
	For each $ P\in \mathcal{C} $,
	define
	\begin{center}
		$ \overline{\psi_P}:\mathrm{dom}(P)\times T\rightarrow P^*E\qquad $ by $\qquad \overline{\psi_P}(r,t)=(\Pr_1,qt)(r,t)=(r,[P,r,t])$.
	\end{center}
	It is clear that $ \overline{\psi_P} $ is well-defined, smooth and fiber preserving.
	To verify that $ \overline{\psi_P} $ is bijective,
	assume that $ (r,[P',s,t']) $ is an element of $ P^*E $. 
	Then $ x:=P(r)=P'(s) $ and we obtain the $ 1 $-simplex $ P'\stackrel{\textbf{s}}{\longleftarrow}\textbf{x}\stackrel{\textbf{r}}{\longrightarrow} P $ on $ \mathcal{C} $.
	Thus
	$ [P',s,t']=[P,r,\varrho\big{(}g_{P\stackrel{\textbf{r}}{\longleftarrow}\textbf{x}\stackrel{\textbf{s}}{\longrightarrow} P'}(0)\big{)}(t')] $.
	This gives rise to a well-defined map
	\begin{center}
		$ \phi_P:P^*E\rightarrow \mathrm{dom}(P)\times T\qquad $ by $\qquad \phi_P(r,[P',s,t'])=\big{(}r,\varrho g_{P'\stackrel{\textbf{s}}{\longleftarrow}\textbf{x}\stackrel{\textbf{r}}{\longrightarrow} P}(0)(t')\big{)} $, 
	\end{center}
	which is indeed the inverse of $ \overline{\psi_P} $. 
	To see $ \overline{\psi_P} $ is a diffeomorphism, it is sufficient to show that $ \overline{\psi_P} $ is a subduction. Let $ (F,[P',F',\alpha]) $ be a plot in $ P^*E $,
	where
	$ \alpha :V\rightarrow T $ is a plot in $ T $, $ F:V\rightarrow \mathrm{dom}(P) $ is a plot in $ \mathrm{dom}(P) $, and $ F':V\rightarrow \mathrm{dom}(P') $ is a plot in $ \mathrm{dom}(P') $.
	Since $ Q:=P'\circ F'=P\circ F $,
	morphisms $ {P'}\stackrel{F'}{\longleftarrow}Q\stackrel{F}{\longrightarrow} P $ are achieved,
	and we can write
	\begin{center}
		$ \psi_P(F~~,~~\varrho\big{(}g_{P\stackrel{F}{\leftarrow}Q\stackrel{F'}{\rightarrow} P'}\big{)}\cdot\alpha)=(F,[P,F,\varrho\big{(}g_{P\stackrel{F}{\leftarrow}Q\stackrel{F'}{\rightarrow} P'}\big{)}\cdot\alpha])=(F,[P',F',\alpha]). $
	\end{center}
	Now consider $ \psi_P:=P_{\#}\circ\overline{\psi_P}:\mathrm{dom}(P)\times T\rightarrow E $ so that 
	the square
	\begin{displaymath}
		\xymatrix{
			\mathrm{dom}(P)\times T \ar[r]^{\qquad\psi_P}\ar[d]_{\Pr_1 } & E \ar[d]^{\pi} \\
			\mathrm{dom}(P) \ar[r]^{\quad P} & X  }
	\end{displaymath} 
	is a pullback in $ \textsf{Diff} $.
	Diffeological \v{C}ech 1-cocycles  $ g_{P_0\stackrel{F_1}{\leftarrow}Q\stackrel{F_0}{\rightarrow} P_1}$ play the role of transitions: In fact,
	for each $ 1 $-simplex $ P_0\stackrel{F_1}{\longleftarrow}Q\stackrel{F_0}{\longrightarrow} P_1 $ on $ \mathcal{C} $
	we have
	\begin{align*}
		\psi_{P_1}\big{(}F_0(r),t\big{)}&=[P_1,F_0(r),t]\\
		&=[P_0,F_1(r),\varrho g_{P_0\stackrel{F_1}{\leftarrow}Q\stackrel{F_0}{\rightarrow} P_1}(r) (t)]\\
		&=\psi_{P_0}\big{(}F_1(r),\varrho g_{P_0\stackrel{F_1}{\leftarrow}Q\stackrel{F_0}{\rightarrow} P_1}(r) (t)\big{)},
	\end{align*}
	for all $ (r,t)\in \mathrm{dom}(Q)\times T $.
\end{proof}

\begin{theorem}\label{the-B}
	For every $ \mathcal{C}\in\mathsf{CGF}(X) $, there is a well-defined map
	\begin{center}
		$ \eta_{\mathcal{C}}:H^1(X,\mathcal{C},G)\rightarrow\mathrm{Bund}_T(X,G) $
	\end{center}
	taking the class of a diffeological \v{C}ech 1-cocycle 
	$ \{g_{P_0\stackrel{F_1}{\leftarrow}Q\stackrel{F_0}{\rightarrow} P_1}\}_{1-\mathrm{simplexes}(\mathcal{C})} $
	to the class of
	its corresponding diffeological fiber bundle, given by Construction \ref{cons-1},
	such that 
	\begin{center}
		$ \eta_{\mathcal{C}}=\eta_{\mathcal{C}'}\circ j(\mathcal{C}',\mathcal{C}) $,
	\end{center}
	for any refinement $  \mathcal{C}'\rightarrow\mathcal{C}  $ of covering generating families.
	In particular, by universal property we obtain a unique map
	\begin{center}
		$ \Theta:\check{H}^1(X,G)\rightarrow\mathrm{Bund}_T(X,G) $
	\end{center}
	with the property that $ \eta_{\mathcal{C}}=\Theta\circ \zeta_{\mathcal{C}} $, where $ \zeta_{\mathcal{C}}:H^1(X,\mathcal{C},G)\rightarrow\check{H}^1(X,G) $  is the map given by the definition of colimit for $ \mathcal{C}\in\mathsf{CGF}(X) $. Furthermore, the map
	$ \Theta $
	is a bijective correspondence.
\end{theorem}

\subsection{Some applications} 
\begin{proposition}\label{p-xtg}
	Diffeological fiber bundles with gauge group $ \mathrm{Diff}(T) $, the action $ \mathrm{id}:\mathrm{Diff}(T)\rightarrow \mathrm{Diff}(T) $, and fiber $ T $ on $ X $  are exactly the usual diffeological fiber bundles with fiber $ T $ on $ X $. Consequently, there exists a bijective correspondence between $ \check{H}^1(X,\mathrm{Diff}(T)) $ and $ \mathrm{Bund}_T(X,\mathrm{Diff}(T))\cong\mathrm{Bund}_{fiber}(X,T) $.
\end{proposition}
\begin{proof}
	Let $ \pi:E\rightarrow X $ be a usual diffeological fiber bundles with fiber $ T $.
	Take  the collection $ \mathcal{C} $ of all of the plots $ P $ in $ X $ for which there is a diffeomorphism
	$ \overline{\psi_P}:\mathrm{dom}(P)\times T\rightarrow P^*E $. All constant plots in $ X $ belong to this collection, so $ \mathcal{C} $ is a parametrized cover of $ X $.
	Let $ Q:V\rightarrow X $ be an arbitrary plot.
	The pullback $ Q^{*}\pi:Q^{*}X\rightarrow V $ is
	locally trivial with the fiber $ T $, by definition.
	Thus, there exists an open cover $ \{V_i\}_{i\in J} $ of $ V $ such that 
	$ (Q|_{V_i})^*\pi $ is trivial, for all $ i\in J $.
	Therefore, $ Q $ is as the supremum of a family of plots belonging to $ \mathcal{C} $, and hence $ \mathcal{C} $ is a covering generating family for $ X $.
	For each $ P\in\mathcal{C} $, $ \psi_P:=P_{\#}\circ\overline{\psi_P} $ satisfies the condition (a) of Definition \ref{def-fbgt}.
	For each $ 1 $-simplex $ P_0\stackrel{F_1}{\longleftarrow}Q\stackrel{F_0}{\longrightarrow} P_1 $ on $ \mathcal{C} $ define
	\begin{center}
		$ g_{P_0\stackrel{F_1}{\leftarrow}Q\stackrel{F_0}{\rightarrow} P_1}:\mathrm{dom}(Q) \rightarrow \mathrm{Diff}(T)\quad $ by $\quad g_{P_0\stackrel{F_1}{\leftarrow}Q\stackrel{F_0}{\rightarrow} P_1}(r) :=(\psi_{P_0,F_1(r)})^{-1}\circ\psi_{P_1,F_0(r)}$.
	\end{center}
	It is not hard to check that 
	$ g_{P_0\stackrel{F_1}{\leftarrow}Q\stackrel{F_0}{\rightarrow} P_1}:\mathrm{dom}(Q) \rightarrow \mathrm{Diff}(T) $
	is a smooth map satisfying the condition (b) of Definition \ref{def-fbgt}.
	The converse is obvious, by Remark \ref{rem-1} and Proposition \ref{P1}. 
\end{proof}

\begin{proposition} 
	Consider
	$ \varrho:G\rightarrow \mathrm{Diff}(G) $ defined by $ \varrho (g)\mapsto L_g $, which is the smooth action of $ G $ on itself induced by the left transformations.
	Then diffeological fiber bundles with gauge group $ G $ and fiber $ G $ on $ X $ are exactly diffeological $ G $-principal bundles on $ X $. As a result, there is a bijective correspondence between $ \check{H}^1(X,G) $ and $ \mathrm{Bund}_{G}(X,G)\cong\mathrm{Bund}_{Prin}(X,G) $. It is obvious that in the case where $ G $ is abelian, there is a bijective correspondence between the $ 1 $st diffeological \v{C}ech cohomology group $ \check{H}^1(X,G) $ and $ \mathrm{Bund}_{Prin}(X,G) $. 
\end{proposition}
\begin{proof}
	Let $ \pi:E\rightarrow X $ be a diffeological $ G $-principal bundle.
	Similar to Proposition \ref*{p-xtg} we can specify trivializations.
	Also, for each $ 1 $-simplex $ P_0\stackrel{F_1}{\longleftarrow}Q\stackrel{F_0}{\longrightarrow} P_1 $ with $ P_0, P_1\in \mathcal{C} $ take
	\begin{center}
		$ g_{P_0\stackrel{F_1}{\leftarrow}Q\stackrel{F_0}{\rightarrow} P_1}:\mathrm{dom}(Q) \rightarrow G\quad $ by $\quad g_{P_0\stackrel{F_1}{\leftarrow}Q\stackrel{F_0}{\rightarrow} P_1}(r) :=(\psi_{P_0,F_1(r)})^{-1}\circ\psi_{P_1,F_0(r)}(1_G)$,
	\end{center}
	where $ 1_G $ is the identity element of $ G $.
	
	Conversely, suppose that  $ \pi:E\rightarrow X $ is a diffeological fiber bundle with gauge group $ G $ and fiber $ G $.
	By 	Theorem \ref{the-B}, $ \pi:E\rightarrow X $ is isomorphic to the corresponding diffeological fiber bundle given by Construction \ref{cons-1} for the transitions of $ \pi:E\rightarrow X $.
	It is not hard to see that the corresponding diffeological fiber bundle is diffeological $ G $-principal bundle with a natural smooth action of $ G $, and hence,
	$ \pi:E\rightarrow X $ is diffeological $ G $-principal bundle.
\end{proof}

\begin{proposition} 
	Suppose that $ V $ is a diffeological vector space and let $ \mathrm{GL}(V) $ denote the space of all linear isomorphisms of diffeological vector spaces.
	Let $ \varrho:\mathrm{GL}(V)\hookrightarrow \mathrm{Diff}(V) $ be the linear action of $ \mathrm{GL}(V) $ on $ V $.
	Then diffeological fiber bundles with gauge group $ \mathrm{GL}(V) $ and the typical fiber $ V $ on $ X $ are the same as diffeological vector bundles of the typical fiber $ V $ on $ X $. Also, there exists a bijective correspondence between $ \check{H}^1(X,\mathrm{GL}(V)) $ and $ \mathrm{Bund}_{V}(X,\mathrm{GL}(V))\cong\mathrm{Bund}_{Vect}(X,V) $.
\end{proposition}
\begin{proof}
	The proof is straightforward.
\end{proof}

\begin{remark} 
	In this way, one can simply suggest diffeological affine bundles by 
	diffeological fiber bundles with the gauge group $ \mathrm{Affin}(A) $, the typical fiber $ A $ a diffeological affine space and the natural action $ \varrho:\mathrm{Affin}(A)\hookrightarrow \mathrm{Diff}(A) $.
	In other words, a diffeological affine space  $ \pi:E\rightarrow X $ over $ X $ is a diffeological affine bundle, if there is a
	covering generating family $ \mathcal{C} $ of $ X $
	such that for every $ P\in\mathcal{C} $, the pullback is trivial of fiber type $ A $ in the sense that 
	there is a fiber preserving diffeomorphism $ \psi_P:\mathrm{dom}(P)\times A\rightarrow P^*E $
	such that for every $ r\in \mathrm{dom}(P) $, the
	restriction $ \psi_{P,r}:A\rightarrow E_{P(r)} $ is an isomorphism of diffeological affine spaces.
	Hence, we have a bijective correspondence between $ \check{H}^1(X,\mathrm{Affin}(A)) $ and $ \mathrm{Bund}_{Affin}(X,\mathrm{Affin}(A)) $.
\end{remark}

\subsection{Proof of Theorem \ref{the-B}}

\begin{proposition}\label{pro-well}
	If two diffeological \v{C}ech 1-cocycles 
	\begin{center}
		$ \{g_{P_0\stackrel{F_1}{\leftarrow}Q\stackrel{F_0}{\rightarrow} P_1}\}_{1-\mathrm{simplexes}(\mathcal{C})} \qquad$ and $ \qquad\{g'_{P_0\stackrel{F_1}{\leftarrow}Q\stackrel{F_0}{\rightarrow} P_1}\}_{1-\mathrm{simplexes}(\mathcal{C})} $ 
	\end{center}
	on a covering generating family $\mathcal{C}$ are equivalent, their corresponding diffeological fiber bundles, given by Construction \ref{cons-1}, are isomorphic.
\end{proposition}
\begin{proof}
	Suppose that two diffeological \v{C}ech 1-cocycles 
	\begin{center}
		$ \{g_{P_0\stackrel{F_1}{\leftarrow}Q\stackrel{F_0}{\rightarrow} P_1}\}_{1-\mathrm{simplexes}(\mathcal{C})} \qquad$ and $ \qquad\{g'_{P_0\stackrel{F_1}{\leftarrow}Q\stackrel{F_0}{\rightarrow} P_1}\}_{1-\mathrm{simplexes}(\mathcal{C})} $ 
	\end{center}
	are equivalent.
	By definition, there exists a family $ \{h_{P}:\mathrm{dom}(P)\rightarrow G\}_{P\in \mathcal{C}} $  of smooth maps such that
	\begin{center}
		$ g'_{P_0\stackrel{F_1}{\leftarrow}Q\stackrel{F_0}{\rightarrow} P_1}(r)\quad(h_{P_1}\circ F_0(r))=(h_{P_0}\circ F_1(r))\quad g_{P_0\stackrel{F_1}{\leftarrow}Q\stackrel{F_0}{\rightarrow} P_1}(r) $.
	\end{center}
	Let $ \pi:E\rightarrow X $ and $ \pi:E'\rightarrow X $ be diffeological fiber bundles corresponding to these diffeological \v{C}ech 1-cocycles according to Construction \ref{cons-1}, respectively.
	Define $ \Phi:E\rightarrow E' $
	by $ \Phi([P,r,t])=[P,r,\varrho h_{P}(r) (t)] $. 
	To see $ \Phi $ is well-defined, let $ [P,r,t]=[P',r',t'] $. Then there are morphisms $ P\stackrel{F}{\longleftarrow}Q\stackrel{F'}{\longrightarrow} P' $ such that $ F(s)=r$ and $ F'(s)=r' $ for some $ s\in \mathrm{dom}(Q)$, and $ t=\varrho g_{P\stackrel{F}{\leftarrow}Q\stackrel{F'}{\rightarrow} P'}(s) (t') $.
	Thus,
	\begin{align*}
		\varrho\big{(}h_{P}(r)\big{)}(t)=\varrho\big{(}h_{P}(F(s))\big{)}(t)&=\varrho\big{(}h_{P}(F(s))\big{)}\circ(\varrho\big{(}g_{P\stackrel{F}{\leftarrow}Q\stackrel{F'}{\rightarrow} P'}(s)\big{)}(t'))\\
		&=\varrho\big{(}h_{P}(F(s)) *\big{(}g_{P\stackrel{F}{\leftarrow}Q\stackrel{F'}{\rightarrow} P'}(s)\big{)}(t'))\\
		&=\varrho\big{(}g'_{P \stackrel{F}{\leftarrow}Q\stackrel{F'}{\rightarrow} P'}(s)*(h_{P'}\circ F'(s))\big{)}(t'))\\
		&=\varrho\big{(}g'_{P \stackrel{F}{\leftarrow}Q\stackrel{F'}{\rightarrow} P'}(s)\circ\varrho(h_{P'}\circ F'(s))\big{)}(t'))\\
		&=\varrho\big{(}g'_{P \stackrel{F}{\leftarrow}Q\stackrel{F'}{\rightarrow} P'}(s)\circ\varrho(h_{P'}(r))\big{)}(t'))\\
	\end{align*}
	Hence, 	$ [P,r,\varrho h_{P}(r) (t)]=[P',r',\varrho h_{P'}(r') (t')] $.
	It is trivial that $ \Phi $ is smooth. Similarly, the map $ \Psi:E'\rightarrow E $
	given by $ \Psi([P,r,t])=[P,r,\varrho h_{P}(r) ^{-1}(t)] $ is the smooth inverse of $ \Phi $.
	So $ \Phi $ is a diffeomorphism.
	Also, 
	\begin{center}
		$ \pi'\circ\Phi([P,r,t])=\pi'([P,r,\varrho h_{P}(r) (t)])=P(r)=\pi([P,r,t]) $.
	\end{center}
	It is easy to check that $ \Phi $, along with $ h_{P}:\mathrm{dom}(P)\rightarrow G $ given above, and the maps $ \psi_P:\mathrm{dom}(P)\times T\rightarrow E $ and $ \psi'_P:\mathrm{dom}(P)\times T\rightarrow E' $ given according to the proof of Proposition \ref{prop-cons}, for each $ P\in\mathcal{C} $, is an isomorphism between $ \pi $ and $ \pi' $.
\end{proof}
By proposition \ref{pro-well}, we get a well-defined map
$ \eta_{\mathcal{C}}:H^1(X,\mathcal{C},G)\rightarrow\mathrm{Bund}_T(X,G) $ for all $ \mathcal{C}\in\mathsf{CGF}(X) $. 
\begin{proposition}
	The map
	$ \eta_{\mathcal{C}}:H^1(X,\mathcal{C},G)\rightarrow\mathrm{Bund}_T(X,G) $ satisfy the condition 
	\begin{center}
		$ \eta_{\mathcal{C}}=\eta_{\mathcal{C}'}\circ j(\mathcal{C}',\mathcal{C}) $,
	\end{center}
	for any refinement $  \mathcal{C}'\rightarrow\mathcal{C}  $ of covering generating families, where $ j(\mathcal{C}',\mathcal{C}):H^1(X,\mathcal{C};A)\rightarrow H^1(X,\mathcal{C}';A) $ is the induced map by the refinement. 
\end{proposition}
\begin{proof}
	Assume that $ \mathcal{C}=\{P_\alpha\}_{\alpha\in I} $ and $ \mathcal{C}'=\{P'_\beta\}_{\beta\in J} $ are covering generating families of $ X $ and 
	let $ \mathcal{C}'\prec\mathcal{C} $ be a refinement together with 
	a refinement map $ \lambda:J\rightarrow I $ and a family $ \{f_{\beta}\}_{\beta\in J} $  of morphisms $ P'_\beta\stackrel{f_\beta}{\longrightarrow} P_{\lambda(\beta)} $ between plots.
	To verify $ \eta_{\mathcal{C}}=\eta_{\mathcal{C}'}\circ j(\mathcal{C}',\mathcal{C}) $, we have to check that the map
	\begin{align*}
		\Xi:\bigsqcup_{\beta\in J}\mathrm{dom}(P'_\beta)\times T/\sim \quad&\longrightarrow\quad\bigsqcup_{\alpha\in I}\mathrm{dom}(P_\alpha)\times T/\sim  \\
		[P'_\beta,r,t]\quad&\longmapsto\quad [P_{\lambda(\beta)},f_\beta(r),t] 
	\end{align*}
	is a diffeomorphism, where the quotient spaces are defined according to Construction \ref{cons-1}.
	Clearly, $ \Xi $ is smooth. To injectivity, let $ \Xi[P'_{\beta_1},r_1,t_1]=\Xi[P'_{\beta_2},r_2,t_2] $ so that $ [P_{\lambda(\beta_1)},f_{\beta_1}(r_1),t_1]=[P_{\lambda(\beta_2)},f_{\beta_2}(r_2),t_2] $.
	Then there are morphisms $ P_{\lambda(\beta_1)}\stackrel{F_{\beta_1}}{\longleftarrow}Q\stackrel{F_{\beta_2}}{\longrightarrow} P_{\lambda(\beta_2)} $ such that $ F_{\beta_1}(s)=f_{\beta_1}(r_1)$ and $ F_{\beta_2}(s)=f_{\beta_2}(r_2) $ for some $ s\in \mathrm{dom}(Q)$, and $ t_1=\varrho g_{P_{\lambda(\beta_1)}\stackrel{F_{\beta_1}}{\longleftarrow}Q\stackrel{F_{\beta_2}}{\longrightarrow} P_{\lambda(\beta_2)}}(s) (t_2) $.
	This implies that
	\begin{center}
		$ x:=P'_{\beta_1}(r_1)=P_{\lambda(\beta_1)}\circ f_{\beta_1}(r_1)=P_{\lambda(\beta_1)}\circ F_{\beta_1}(s)=P_{\lambda(\beta_2)}\circ F_{\beta_2}(s)=P_{\lambda(\beta_2)}\circ f_{\beta_2}(r_2)=P'_{\beta_2}(r_2) $
	\end{center}
	and we get
	\begin{align*}
		\varrho g_{P'_{\beta_1}\stackrel{\textbf{r}}{\longleftarrow}\textbf{x}\stackrel{\textbf{r}'}{\longrightarrow} P'_{\beta_2}}(0) (t_2)&=\varrho g_{P_{\lambda(\beta_1)}\stackrel{ }{\longleftarrow}\textbf{x}\stackrel{ }{\longrightarrow} P_{\lambda(\beta_2)}}(0) (t_2)\\
		&=\varrho g_{P_{\lambda(\beta_1)}\stackrel{F_{\beta_1}}{\longleftarrow}Q\stackrel{F_{\beta_2}}{\longrightarrow} P_{\lambda(\beta_2)}}(s) (t_2)=t_1 
	\end{align*}
	Therefore, $ [P'_{\beta_1},r_1,t_1]=[P'_{\beta_2},r_2,t_2] $.
	Now it is sufficient that $ \Xi $ is a subduction.
	Let $ Q $ is a plot in $ \bigsqcup_{\alpha\in I}\mathrm{dom}(P_\alpha)\times T/\sim $ and $ r\in\mathrm{dom}(Q) $. Then
	there are a open neighborhood $ U $ of $ r $, an $ \alpha\in I $, plots $ G $  and $ R $ in $ \mathrm{dom}(P_\alpha) $ and $ T $, respectively, such that
	$ Q|_U=[P_{\alpha}, G, R] $. Because $ \mathcal{C}' $ is a covering generating family for $ X $, 
	there are a open neighborhood $ V\subseteq \mathrm{dom}(P_\alpha) $ of $ G(r) $, a plot $ P'_{\beta}\in\mathcal{C}' $,  and a smooth map $ F  $ between domains such that
	$ P_{\alpha}|_V=P'_{\beta}\circ F $, which gives rise to the $ 1 $-simplex
	$ P_{\alpha} \stackrel{\imath}{\longleftarrow}P_{\alpha}|_V\stackrel{f_\beta\circ F}{\longrightarrow} P_{\lambda(\beta)} $.
	Then
	$ [P'_{\beta}, F, \varrho g_{P_{\alpha} \stackrel{\imath}{\longleftarrow}P_{\alpha}|_V\stackrel{f_\beta\circ F}{\longrightarrow} P_{\lambda(\beta)}}\circ G\cdot R] $ is a plot in $ \bigsqcup_{\beta\in J}\mathrm{dom}(P'_\beta)\times T/\sim $ such that 
	\begin{align*}
		\Xi\circ[P'_{\beta}, F, \varrho g_{P_{\alpha} \stackrel{\imath}{\longleftarrow}P_{\alpha}|_V\stackrel{f_\beta\circ F}{\longrightarrow} P_{\lambda(\beta)}}\circ G\cdot R]
		&=[P_{\lambda(\beta)}, f_\beta\circ F, \varrho g_{P_{\alpha} \stackrel{\imath}{\longleftarrow}P_{\alpha}|_V\stackrel{f_\beta\circ F}{\longrightarrow} P_{\lambda(\beta)}}\circ G\cdot R]\\
		&=[P_{\alpha}, G, R]=Q|_U .
	\end{align*}
	Similarly, one can check that $ \Xi $ is surjective so that it is a subduction.
\end{proof}
By universal property, we obtain a unique map
\begin{center}
	$ \Theta:\check{H}^1(X,G)\rightarrow\mathrm{Bund}_T(X,G) $
\end{center}
with the property that $ \eta_{\mathcal{C}}=\Theta\circ \zeta_{\mathcal{C}} $, where $ \zeta_{\mathcal{C}}:H^1(X,\mathcal{C},G)\rightarrow\check{H}^1(X,G) $  is the map given by the definition of colimit for $ \mathcal{C}\in\mathsf{CGF}(X) $.

On the other hand, one can observe that transititions of a diffeological fiber bundle with gauge group behaves like diffeological \v{C}ech 1-cocycles with values in the sheaf $ G $ (see Example \ref{exa-shv-g}). In fact,
\begin{enumerate}
	\item[$ (a) $]
	$ g_{P\stackrel{F}{\leftarrow}Q\stackrel{F}{\rightarrow} P}(r) $ is the constant map with the value $ 1_G $, where $ 1_G $ denotes the identity element of $ G $, 
	\item[$ (b) $]
	$ g_{P_0\stackrel{F_1}{\leftarrow}Q\stackrel{F_0}{\rightarrow} P_1}(r) $ is the inverse of $ g_{P_1\stackrel{F_0}{\leftarrow}Q\stackrel{F_1}{\rightarrow} P_0}(r) $ in 
	$ G $, for all $ r\in \mathrm{dom}(Q) $.
	
	\item[$ (c) $]
	If
	
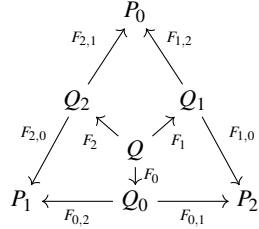
\begin{center}
		\begin{tikzpicture}
			\matrix (m) [matrix of math nodes, row sep=.8em,
			column sep=0.6em, text height=1.2ex, text depth=0.25ex]
			{ & & P_0 & & \\
				& &  & & \\
				& &  & & \\
				& Q_2  & & Q_1 & \\
				& & Q & & \\
				P_1 & & Q_0 & & P_2 \\};
			\path[->,font=\scriptsize]
			(m-4-2) edge node[auto] {$ F_{2,1} $} (m-1-3) edge node[above] {$F_{2,0}\qquad$} (m-6-1)
			(m-4-4) edge node[above] {$\qquad F_{1,2} $} (m-1-3) edge node[auto] {$F_{1,0} $} (m-6-5)
			(m-6-3) edge node[auto] {$ F_{0,2} $} (m-6-1) edge node[below] {$F_{0,1} $} (m-6-5)
			(m-5-3) edge node[auto]  {$ F_2 $} (m-4-2) edge node[below]  {$\qquad F_1 $} (m-4-4) edge node[auto]  {$ F_0 $} (m-6-3);
		\end{tikzpicture} 
	\end{center}
	is a $2$-simplex on $ \mathcal{C} $,	then 
	\begin{center}
		$ g_{P_0\stackrel{F_{2,1}}{\leftarrow}Q_2\stackrel{F_{2,0}}{\rightarrow} P_1}(F_2(r))\quad g_{P_1\stackrel{F_{0,2}}{\leftarrow}Q_0\stackrel{F_{0,1}}{\rightarrow} P_2}(F_0(r))=g_{P_0\stackrel{F_{1,2}}{\leftarrow}Q_1\stackrel{F_{1,0}}{\rightarrow} P_2}(F_1(r)) $	
	\end{center}
	for all $ r\in \mathrm{dom}(Q) $. 
\end{enumerate}
This suggests to define the map 
$ \Delta:\mathrm{Bund}_T(X,G)\rightarrow \check{H}^1(X,G) $ 
taking the class of a diffeological fiber bundle with gauge group to the class of its transititions
as diffeological \v{C}ech 1-cocycles.  		
\begin{proposition}
	The map 
	$ \Delta:\mathrm{Bund}_T(X,G)\rightarrow \check{H}^1(X,G) $
	is well-defined.
\end{proposition}
\begin{proof}
	Let $ \pi:E\rightarrow X $ and $ \pi':E'\rightarrow X $ be
	diffeological fiber bundle with gauge group $ G $ 
	and typical fiber $ T $ on $ X $ with the data given as Definition \ref{def-fbgt}.  
	Assume that $ \Phi $ is an isomorphism between $ \pi $ and $ \pi' $ along with
	a common refinement $ \mathcal{B}=\{R_\beta\}_{\beta\in J} $ of $ \mathcal{C}=\{P_\alpha\}_{\alpha\in I} $ and $ \mathcal{C}'=\{P'_{\alpha'}\}_{\alpha'\in I'} $,  
	also refinement maps $ \lambda:J\rightarrow I $ and $ \lambda':J\rightarrow I' $, a family
	$ \{P_{\lambda(\beta)}\stackrel{f_\beta}{\longleftarrow}R_\beta \stackrel{f'_\beta}{\longrightarrow} P'_{\lambda'(\beta)}\}_{\beta\in J} $ 
	of morphisms of plots,
	such that 
	for each $ \beta\in J $ there exists a smooth map $ h_{\beta}:\mathrm{dom}(R_\beta)\rightarrow G $ such that
	\begin{center}
		$ \Phi\circ\psi_{P_{\lambda(\beta)}}\big{(}f_\beta(r),t\big{)}=\psi'_{P'_{\lambda'(\beta)}}\Big{(}f'_\beta(r),\varrho h_{\beta}(r) (t)\Big{)}, $
	\end{center}	
	for all $ (r,t)\in \mathrm{dom}(R_\beta)\times T $.  
	
	To complete the proof, it is enough to show that $ \lambda^*g$ and $\lambda'^*g' $ are equivalent on $ \mathcal{B} $, where $ g $ and $ g' $ denote transitions of $ \pi $ and $ \pi' $, respectively.
	Let
	$ R_{\beta_0}\stackrel{F_{\beta_1}}{\longleftarrow}Q\stackrel{F_{\beta_0}}{\longrightarrow} R_{\beta_1} $
	be a $ 1 $-simplex on $ \mathcal{B} $. 
	Then the diagram
	
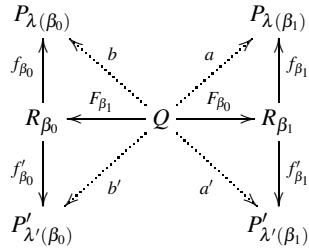
\begin{displaymath}
		\xymatrix{
			P_{\lambda(\beta_0)} & & P_{\lambda(\beta_1)}\\
			R_{\beta_0}\ar[d]_{f'_{\beta_0}}\ar[u]^{f_{\beta_0}} & \ar[l]_{F_{\beta_1}} \ar@{.>}[ld]^{b'}\ar@{.>}[lu]_{b} Q \ar@{.>}[rd]_{a'}\ar@{.>}[ru]^{a}\ar[r]^{F_{\beta_0}} & R_{\beta_1} \ar[d]^{f'_{\beta_1}}\ar[u]_{f_{\beta_1}} \\
			P'_{\lambda'(\beta_0)} & & P'_{\lambda'(\beta_1)}
		}
	\end{displaymath}
	commutes, so $ 1 $-simplexes 	$ P_{\lambda(\beta_0)}\stackrel{b}{\longleftarrow}Q\stackrel{a}{\longrightarrow} P_{\lambda(\beta_1)} $ 
	and
	$ P'_{\lambda'(\beta_0)}\stackrel{b'}{\longleftarrow}Q\stackrel{a'}{\longrightarrow} P'_{\lambda'(\beta_1)} $ are achieved
	on $ \mathcal{C} $ and $ \mathcal{C}' $, respectively. Thus, for each $ r\in\mathrm{dom}(Q) $, we can write
	\begin{align*}
		\varrho \lambda'^*g'_{R_{\beta_0}\stackrel{F_{\beta_1}}{\longleftarrow}Q\stackrel{F_{\beta_0}}{\longrightarrow} R_{\beta_1}}(r)\circ\varrho(h_{\beta_1}\circ F_{\beta_0}(r))
		&=\varrho  g'_{P'_{\lambda'(\beta_0)}\stackrel{b'}{\longleftarrow}Q\stackrel{a'}{\longrightarrow} P'_{\lambda'(\beta_1)}}(r)\circ\big{(}\psi'_{P'_{\lambda'(\beta_1)},a'(r)}\big{)}^{-1}\circ\Phi\circ\psi_{P_{\lambda(\beta_1)},a(r)}\\
		&=(\psi'_{P_{\lambda'(\beta_0)},b'(r)})^{-1}\circ\Phi\circ\psi_{P_{\lambda(\beta_1)},a(r)}\\
		&=(\psi'_{P_{\lambda'(\beta_0)},b'(r)})^{-1}\circ\Phi\circ\psi_{\lambda(\beta_0),b(r)}\circ\varrho g_{P_{\lambda(\beta_0)}\stackrel{b}{\longleftarrow}Q\stackrel{a}{\longrightarrow} P_{\lambda(\beta_1)}}(r)\\
		&=\varrho(h_{\beta_0}\circ F_{\beta_1}(r))\circ \varrho g_{P_{\lambda(\beta_0)}\stackrel{b}{\longleftarrow}Q\stackrel{a}{\longrightarrow} P_{\lambda(\beta_1)}}(r)\\
		&=\varrho(h_{\beta_0}\circ F_{\beta_1}(r))\circ \varrho \lambda^*g_{R_{\beta_0}\stackrel{F_{\beta_1}}{\longleftarrow}Q\stackrel{F_{\beta_0}}{\longrightarrow} R_{\beta_1}}(r). \\
	\end{align*}
	Since $ \varrho $ is a faithful action, we conclude that 
	\begin{center}
		$ \lambda'^* g'_{R_{\beta_0}\stackrel{F_{\beta_1}}{\longleftarrow}Q\stackrel{F_{\beta_0}}{\longrightarrow} R_{\beta_1}}(r)\quad h_{\beta_1}\circ F_{\beta_0}(r)
		=h_{\beta_0}\circ F_{\beta_1}(r)\quad \lambda^*g_{R_{\beta_0}\stackrel{F_{\beta_1}}{\longleftarrow}Q\stackrel{F_{\beta_0}}{\longrightarrow} R_{\beta_1}}(r) $.
	\end{center}
	Therefore
	$  \mathsf{class}(g) $ 
	and
	$ \mathsf{class}(g') $
	are the same in
	$ \check{H}^1(X,G) $.
\end{proof}

\begin{proposition}\label{prop-sur}
	Suppose that $ \pi:E\rightarrow X $ is
	diffeological fiber bundle with gauge group $ G $ 
	and typical fiber $ T $ on $ X $.  
	Then $ \pi $ is isomorphic to diffeological fiber bundle
	with gauge group $ G $ corresponding to the transitions of $ \pi $, according to Construction \ref{cons-1}.
\end{proposition}
\begin{proof} 
	Assume that $ \pi:E\rightarrow X $ is a diffeological fiber bundle with the data given as Definition \ref{def-fbgt}.
	Let $ \Pr:\bigsqcup_{P\in\mathcal{C}}\mathrm{dom}(P)\times T/\sim\rightarrow X $ be the diffeological fiber bundle corresponding to the transitions of $ \pi $, according to Construction \ref{cons-1}.
	Consider the well-defined map $ \Phi:\bigsqcup_{P\in\mathcal{C}}\mathrm{dom}(P)\times T/\sim\rightarrow E $
	with $ \Phi([P,r,t])=\psi_{P}(r,t) $. Then the diagram
	\begin{displaymath}
		\xymatrix{
			& \bigsqcup_{P\in\mathcal{C}}\mathrm{dom}(P)\times T\ar[drr]^{\quad \bigsqcup_{P\in\mathcal{C}} \psi_P}\ar[dl]_{q} & &\\
			\bigsqcup_{P\in\mathcal{C}}\mathrm{dom}(P)\times T/\sim \ar[dr]_{\Pr} \ar[rrr]^{\qquad \Phi} & & & E \ar[dll]^{\pi}\\
			& X & &
		}
	\end{displaymath}
	commutes. So to see that $ \Phi $ is an isomorphism, we have to show that it is a diffeomorphism.
	%

	We claim that $ \bigsqcup_{P\in\mathcal{C}} \psi_P:\bigsqcup_{P\in\mathcal{C}}\mathrm{dom}(P)\times T\rightarrow E$ is a subduction, which implies that $ \Phi $ is a subduction by \cite[\S 1.51]{PIZ2013}.
	To see $ \bigsqcup_{P\in\mathcal{C}} \psi_P $ is surjective, take $ \xi \in E $. Since $ \mathcal{C} $ is a parametrized cover of $ X $, there are $ P\in\mathcal{C} $ and $ r\in \mathrm{dom}(P) $ such that $ \pi(\xi)=P(r) $. So $ \xi $ is an element of $ E_{P(r)} $ and we can find a $ t\in T $ with $ \psi_{P,r}(r)=\xi $, or $ \psi_{P}(r,t)=\xi $. Now let $ Q:V\rightarrow E $ be an arbitrary plot in $ E $ and $ r\in V $. Since  $ \mathcal{C} $ is a covering generating family of $ X $ and $ \pi\circ Q $ is a plot in $ X $, we can write $ \pi\circ Q|_{V'}=P\circ F $
	for some $ P\in\mathcal{C} $ and a smooth map $ F:V'\rightarrow \mathrm{dom}(P) $ defined on an open neighborhood $ V'\subseteq V $ of $ r $.
	Then $ (F,Q|_{V'}) $ is a plot in $ P^*E $.
	Set $ L:=(\overline{\psi_P})^{-1}\circ(F,Q|_{V'}) $ so that $ (\bigsqcup_{P\in\mathcal{C}} \psi_P)\circ L=Q|_{V'} $.
	
	To complete the proof, it is sufficient to prove that $ \Phi $ is injective.
	Let	$ \Phi([P,r,t])=\Phi([P',r',t']) $ so that $ \psi_{P}(r,t)=\psi_{P'}(r',t') $ or
	$ \psi_{P,r}(t)=\psi_{P',r'}(t') $.
	Then $ x:=P(r')=P(r') $ and the $ 1 $-simplex $ P'\stackrel{\textbf{r}'}{\longleftarrow}\textbf{x}\stackrel{\textbf{r}}{\longrightarrow} P $ on $ \mathcal{C} $ is achieved.
	By Remark \ref{rem-1}, 
	\begin{center}
		$ \varrho g_{P'\stackrel{\textbf{r}'}{\longleftarrow}\textbf{x}\stackrel{\textbf{r}}{\longrightarrow} P}(0)(t')=(\psi_{P,r})^{-1}\circ\psi_{P',r'}(t')=(\psi_{P,r})^{-1}\circ\psi_{P,r}(t)=t. $
	\end{center}
	Therefore, $ [P,r,t]=[P',r',t'] $.
\end{proof}
\begin{proposition}
	The map 
	$ \Theta:\check{H}^1(X,G)\rightarrow\mathrm{Bund}_T(X,G) $
	is a bijective correspondence.
\end{proposition}
\begin{proof}
	In view of Propositions \ref{prop-cons} and \ref{prop-sur}, it is straightforward to see that
	$ \Delta $ is the inverse of $ \Theta $.
\end{proof}

\appendix
\section{Other versions of \v{C}ech cohomology for diffeological spaces}\label{A}
\subsection{\v{C}ech cohomology on the site of D-open subsets} 

In general, \v{C}ech cohomology could be developed to (pre)sheaves on Grothendieck sites (see, e.g., \cite{T}).
For a diffeological space $ X $, we can consider two natural sites D-$\textbf{Open}(X) $ of D-open subsets and $ \textbf{Plots}(X) $ of plots.
\v{C}ech cohomology on the site $ \textbf{Plots}(X) $ gives rise to the presheaf $ \check{H}^{\bullet} $ assigning $ P\mapsto \check{H}^{\bullet}(\mathrm{dom}(P);A_P)  $, where $ \check{H}^{\bullet}(\mathrm{dom}(P);A_P) $ is the usual \v{C}ech cohomology with coefficients in $ A_P $ on  $ \mathrm{dom}(P) $.
However, \v{C}ech cohomology on the site D-$\textbf{Open}(X) $ is the same as the usual \v{C}ech cohomology with coefficients in $ \Sigma A $ on D-topological space $ X $.
In this appendix, we show that \v{C}ech cohomology on the site D-$\textbf{Open}(X) $ is the same as \v{C}ech cohomology in the framework of provided in Section \ref{S3}, which has been suggested in \cite{DA2018}.

Let $ \mathcal{O}=\{O_{\alpha}\}_{\alpha\in I} $ be a D-open covering of a diffeological space $ X $ and let $ A $ be a presheaf on $ X $. Then $ \mathcal{O}_P=P^{-1} \mathcal{O}=\{P^{-1}(O_{\alpha}) \mid \alpha\in I\} $ is a covering of the domain of definition of any plot $ P $ in $ X $.
Set
\begin{center}
	$ P_{\alpha_0,\dots,\alpha_n}:=P|_{P^{-1}(O_{\alpha_0}\cap\cdots\cap O_{\alpha_n})}=P|_{P^{-1}(O_{\alpha_0})}\times_P\cdots\times_P P|_{P^{-1}(O_{\alpha_n})} $
\end{center}
for every $ n $-simplexes $ O_{\alpha_0}\cap\cdots\cap O_{\alpha_n} $ of elements of $ \mathcal{O} $.
In particular, a plot $ P $ is the supremum of the family $ \{P_{\alpha}\}_{\alpha\in I} $.

For every nonnegative integer $ n $, define the presheaf $ \check{C}^n(\mathcal{O};A) $ assigning to each plot $ P $ in $ X $, the \v{C}ech $ n $-cochain group 
\begin{center}
	$ \check{C}^n(\mathcal{O};A)(P):=\displaystyle\prod_{n-\mathrm{simplexes}} A(P_{\alpha_0,\dots,\alpha_n})$, 
\end{center}
subordinate to the cover $ \mathcal{O}_P $ on $ \mathrm{dom}(P) $ with
coefficients in the presheaf $ A(P) $, and to each morphism $ Q\stackrel{F}{\longrightarrow} P $, the pullback homomorphism $ F^{\star}:\check{C}^n(\mathcal{O}_P;A(P))\rightarrow\check{C}^n(\mathcal{O}_Q;A(Q)) $ that takes $ c $ to $ F^{\star}(c) $
such that
\begin{center}
	$ F^{\star}(c)\big{(} Q_{\alpha_0,\dots,\alpha_n}\big{)} :=F_{\alpha_0,\dots,\alpha_n}^*\big{(} c(P_{\alpha_0,\dots,\alpha_n})\big{)}$.
\end{center}
Also, 
consider the morphisms $ d:\check{C}^n(\mathcal{O};A)\rightarrow\check{C}^{n+1}(\mathcal{O};A) $ of
the \v{C}ech coboundary operators on the domains of plots.
Hence one obtains 
the associated cochain complex
\begin{center}
	$ 0\stackrel{}{\longrightarrow}  \Sigma  \check{C}^0(\mathcal{O};A)(X)\stackrel{\Sigma d}{\longrightarrow} \Sigma  \check{C}^1(\mathcal{O};A)(X)\stackrel{\Sigma d}{\longrightarrow} \Sigma  \check{C}^2(\mathcal{O};A)(X)\stackrel{\Sigma d}{\longrightarrow}\cdots $
\end{center}
of groups of sections and the cohomology groups 
$ H^k(\Sigma  \check{C}^{\bullet}(\mathcal{O};A))  $.

Moreover, a sheaf $ A $ induces an ordinary sheaf $ \Sigma A $ on D-topological space $ X $ by \cite[Theorem 4.2]{DA}, and sheaves  $ A_P $ on $ \mathrm{dom}(P) $ for plots $ P $ in $ X $ by definition.
\begin{proposition} 
	$ \Big{(}\check{C}^{\bullet}(\mathcal{O};\Sigma A),d_X\Big{)} $
	and
	$ \Big{(}\Sigma  \check{C}^{\bullet}(\mathcal{O};A)(X),\Sigma d\Big{)} $  
	are isomorphic.
\end{proposition}
\begin{proof}
	Consider the homomorphism  $\Theta:\check{C}^n(\mathcal{O};\Sigma A)\rightarrow\Sigma\check{C}^n(\mathcal{O};A)(X) $
	taking $ c\in \displaystyle\prod_{n-\mathrm{simplexes}} \Sigma A(O_{\alpha_0,\dots,\alpha_n}) $
	to $ \Theta(c)\in \Sigma\check{C}^n(\mathcal{O};A)(X) $ defined by 
	\begin{center}
		$ \Theta(c)(P)(\alpha_0,\dots,\alpha_n) :=c(\alpha_0,\dots,\alpha_n)(P_{\alpha_0,\dots,\alpha_n}), $ 
	\end{center}
	for every plot $ P $ in $ X $.
	To see that $ \Theta(c) $ is well-defined, let $ Q\stackrel{F}{\longrightarrow} P $ be a morphism of plots in $ X $.
	Then 
	\begin{align*}
		F^{\star}\big{(}\Theta(c)(P)\big{)}(\alpha_0,\dots,\alpha_n)&=F_{\alpha_0,\dots,\alpha_n}^*\big{(} \Theta(c)(P)(\alpha_0,\dots,\alpha_n)\big{)}\\
		&=F_{\alpha_0,\dots,\alpha_n}^*\big{(} c(\alpha_0,\dots,\alpha_n)(P_{\alpha_0,\dots,\alpha_n})\big{)}\\
		&=c(\alpha_0,\dots,\alpha_n)(Q_{\alpha_0,\dots,\alpha_n})\\
		&=\Theta(c)(Q)(\alpha_0,\dots,\alpha_n),\\
	\end{align*}
	and so $ F^{\star}\big{(}\Theta(c)(P)\big{)}=\Theta(c)(Q)$.
	
	On the other hand, 	consider the map  $\Delta:\Sigma\check{C}^n(\mathcal{O};A)(X)\rightarrow\check{C}^n(\mathcal{O};\Sigma A) $
	taking $ \sigma\in \Sigma\check{C}^n(\mathcal{O};A)(X) $
	to $ \Delta(\sigma)\in \check{C}^n(\mathcal{O};\Sigma A) $ defined by 
	$ \Delta(\sigma)(\alpha_0,\dots,\alpha_n)(P) :=\sigma(P)(\alpha_0,\dots,\alpha_n), $ 
	for every plot $ P $ in $ O_{\alpha_0,\dots,\alpha_n} $.
	To observe that $ \Delta(\sigma) $ is well-defined, let $ Q\stackrel{F}{\longrightarrow} P $ be a morphism of plots in $ O_{\alpha_0,\dots,\alpha_n} $.
	We get 
	\begin{align*}
		F^*\big{(}\Delta(\sigma)(\alpha_0,\dots,\alpha_n)(P)\big{)}&=	F^*\big{(}\sigma(P)(\alpha_0,\dots,\alpha_n)\big{)}\\
		&=F^{\star} \big{(} \sigma(P)\big{)}(\alpha_0,\dots,\alpha_n)\\
		&=\sigma(Q)(\alpha_0,\dots,\alpha_n)\\
		&=\Delta(\sigma)(\alpha_0,\dots,\alpha_n)(Q).\\
	\end{align*}
	
	For every $ c\in \displaystyle\prod_{n-\mathrm{simplexes}} \Sigma A(O_{\alpha_0,\dots,\alpha_n}) $, $ (\alpha_0,\dots,\alpha_n) $ and for each plot $ P $ in $ O_{\alpha_0,\dots,\alpha_n} $, we have
	\begin{align*}
		\Delta\big{(}\Theta(c)\big{)}(\alpha_0,\dots,\alpha_n)(P)&=\Theta(c)(P)(\alpha_0,\dots,\alpha_n)\\
		&= c(\alpha_0,\dots,\alpha_n)(P_{\alpha_0,\dots,\alpha_n})\\
		&=c(\alpha_0,\dots,\alpha_n)(P)\\
	\end{align*}
	so that $ \Delta\circ\Theta=\mathrm{id} $.
	To check that $ \Theta\circ\Delta=\mathrm{id} $, for every $ \sigma\in \Sigma\check{C}^n(\mathcal{O};A)(X) $, $ (\alpha_0,\dots,\alpha_n) $ and for each plot $ P $ in $ X $, we get
	\begin{align*}
		\Theta\big{(}\Delta(\sigma)\big{)}(P)(\alpha_0,\dots,\alpha_n)&=\Delta(\sigma)(\alpha_0,\dots,\alpha_n)(P_{\alpha_0,\dots,\alpha_n})\\
		&= \sigma(P_{\alpha_0,\dots,\alpha_n})(\alpha_0,\dots,\alpha_n)\\
		&=\sigma(P)\upharpoonright{P_{\alpha_0,\dots,\alpha_n}}(\alpha_0,\dots,\alpha_n)\\
		&=\sigma(P)(\alpha_0,\dots,\alpha_n),\\
	\end{align*}
	Therefore, $ \Theta $ is an isomorphism.
	
	Finally, to verify that	$ \Sigma d \circ\Theta=\Theta\circ d_X $, let $ c\in \displaystyle\prod_{n-\mathrm{simplexes}} \Sigma A(O_{\alpha_0,\dots,\alpha_n}) $. Then
	\begin{align*}
		\Sigma d \big{(}\Theta(c)\big{)}(P)(\alpha_0,\dots,\alpha_{n+1})
		&= d_P \big{(}\Theta(c)(P)\big{)}(\alpha_0,\dots,\alpha_{n+1})\\
		&= \sum_{k=0}^{n+1}(-1)^k~~\Theta(c)(P)(\alpha_0,\dots,\widehat{\alpha_k},\dots,\alpha_{n+1})\upharpoonright{P_{\alpha_0,\dots,\alpha_{n+1}}}\\
		&= \sum_{k=0}^{n+1}(-1)^k~~c(\alpha_0,\dots,\widehat{\alpha_k},\dots,\alpha_{n+1})(P_{\alpha_0,\dots,\widehat{\alpha_k},\dots,\alpha_{n+1}})\upharpoonright{P_{\alpha_0,\dots,\alpha_{n+1}}}\\
		&= \sum_{k=0}^{n+1}(-1)^k~~c(\alpha_0,\dots,\widehat{\alpha_k},\dots,\alpha_{n+1})(P_{\alpha_0,\dots,\alpha_{n+1}})\\
		&=\big{(}d_X(c)\big{)}(\alpha_0,\dots,\alpha_{n+1})(P_{\alpha_0,\dots,\alpha_{n+1}})\\
		&=\Theta\big{(}d_X(c)\big{)}(P)(\alpha_0,\dots,\alpha_{n+1}).\\
	\end{align*}
\end{proof}

\begin{corollary} 
	If $ A $ is a sheaf on a diffeological space $X$, then the groups $ H^0\big{(}\Sigma  \check{C}^{\bullet}(\mathcal{O};A)\big{)} $ and $\Sigma A(X) $ are isomorphic.
\end{corollary}

\subsection{Basic sections and the \v{C}ech cohomology due to Krepski-Watts-Wolbert} 
Here we show that for a sheaf $ A $, the $ 0 $th \v{C}ech cohomology group $ \check{H}_{KWW}^{0}(\mathcal{C};A) $  due to Krepski-Watts-Wolbert in \cite{KWW}, the group of sections of $ A $, and the group of basic sections $ \Sigma_{basic} ev^*A(\textsf{Nebula}(\mathcal{C})) $   are all isomorphic.

\begin{definition}
	Any smooth map $ f:X\rightarrow Y $ induces a morphism of sites $ \overline{f}:X_{\mathsf{Plots}}\rightarrow Y_{\mathsf{Plots}}$ by taking $ Q\stackrel{F}{\longrightarrow} P $ to $ f\circ Q\stackrel{F}{\longrightarrow} f\circ P $. 
	Let $ A $ be a presheaf on diffeological space $ Y $. 
	The \textbf{pullback} $ f^*A $ of the presheaf $ A $ by $ f $ on $ X $ is given by $ f^*A:=A\circ\overline{f} $.
	This gives rise to a natural map 	between sections:
	\begin{center}
		$ f^*:\Sigma A(Y) \rightarrow \Sigma f^*A(X),\quad (f^*\sigma)(P)=\sigma(f\circ P) $ 
	\end{center}
	
\end{definition}
\begin{definition}
	Let $ A $ be a presheaf on diffeological space $ Y $ and let $ f:X\rightarrow Y $ be a smooth map. 
	A section $ \sigma $ of $ f^*A $ is \textbf{basic} if we have the following implication for all plots $ P $ and $ P' $ in $ X $:
	\begin{center}
		$ f\circ P=f\circ P'\quad\Longrightarrow\quad \sigma(P)=\sigma(P'). $ 
	\end{center} 
	We denote by  $ \Sigma_{basic} f^*A(X) $ the set of basic sections of $ f^*A $.
\end{definition}

\begin{proposition}\label{p-KWW}
	Let $ A $ be a sheaf of abelian groups on diffeological space $ Y $. 
	If $ f:X\rightarrow Y $ is a subduction,
	then $ \Sigma_{basic} f^*A(X) $ and $ \Sigma A(Y) $ are isomorphic.
\end{proposition}
\begin{proof}
	First, we prove that the homomorphism $ f^*:\Sigma A(Y) \rightarrow \Sigma f^*A(X) $ is injective.
	Let $ f^*\sigma=f^*\sigma' $ and $ P $ be any plot in $ Y $. There is a plot $ L $ in $ X $ such that $ P=f\circ L $.
	Then 
	\begin{center}
		$ \sigma(P)=\sigma(f\circ L)=(f^*\sigma)(L)=(f^*\sigma')(L)=\sigma'(f\circ L)=\sigma'(P) $
	\end{center}
	and $ \sigma=\sigma' $.
	Thus, $ f^* $ is an isomorphism on its image. We now show that $ Im(f^*)=\Sigma_{basic} f^*A(X) $. 
	It is obvious that $ Im(f^*)\subseteq\Sigma_{basic} f^*A(X) $, so let $ \sigma $ be a basic section of $ f^*A $.
	Let $ P $ be a plot in $ Y $.	Since $ f $ is a subduction, $ P $
	is as the supremum of a compatible family 
	$\lbrace P_i \rbrace_{i\in J}$ such that $ P_i=f\circ L_i $ for some plot $ L_i:U_i\rightarrow X $ in $ X $.
	Set 
	\begin{center}
		$ \tau(P_i):=\sigma(L_i)\in f^*A (L_i)=A(P_i) $. 
	\end{center}
	Because $ \sigma $ is a section, for every $i, j\in J$, we have
	\begin{align*}
		\tau(P_i)\upharpoonright{P_{ij}}=\imath_i^*(\tau(P_i))=\imath_i^*(\sigma(L_i))
		=\imath_j^*(\sigma(L_j))=\imath_j^*(\tau(P_j))=\tau(P_j)\upharpoonright{P_{ij}}
	\end{align*}
	By sheaf property, there exists an element $ \tau(P)\in A(P) $ with $ \tau(P)\upharpoonright{P_i}=\tau(P_i) $, for every $i\in J$.
	By the basic condition, $ \tau(P) $ is independent of choice of compatible family 
	$\lbrace P_i \rbrace_{i\in J}$ with  the supremum $ P $.
	Also, one can see that
	\begin{center}
		$ F^*\tau(P)=F^*\sigma(L)=\sigma(L\circ F)=\tau(f\circ L\circ F)=\tau(P\circ F) $.
	\end{center}
	Hence $ \tau $ is a well-defined section of $ A $ such that $ f^*(\tau) = \sigma$. 
\end{proof}
\begin{proposition}
	Suppose $ A $ is a sheaf of abelian groups on $ X $ and let $\mathcal{C}$ be a covering generating family of $ X $.
	Then the groups $ \check{H}_{KWW}^{0}(\mathcal{C};A) $, $ \Sigma_{basic} ev^*A(Nebula(\mathcal{C})) $, and $ \Sigma A(X) $ are all isomorphic.
\end{proposition}
\begin{proof}
	$ \check{H}_{KWW}^{0}(\mathcal{C};A) $ is isomorphic to
	$ \ker\partial_0:\check{C}_{KWW}^{0}(\mathcal{C};A)\rightarrow\check{C}_{KWW}^{1}(\mathcal{C};A) $, where $ \partial_0 $ denotes the $ 0 $th coboundary operator.
	On the other hand, $ \ker\partial_0 $ is exactly $ \Sigma_{basic} ev^*A(Nebula(\mathcal{C})) $ and 
	since
	\begin{center}
		$ ev:\textsf{Nebula}(\mathcal{C})\rightarrow X, ~~~(P,r)\mapsto P(r) $ 
	\end{center}
	is a subduction,
	$ \Sigma_{basic} ev^*A(Nebula(\mathcal{C})) $
	is isomorphic to $ \Sigma A(X) $ by Proposition \ref{p-KWW}.
\end{proof}


\begin{thebibliography}{99}
	\bibitem{ARA} A. Ahmadi, \textit{Submersions, immersions, and \'{e}tale maps in diffeology},  arXiv:2203.05994 (2022). 
	\bibitem{AD}
	A. Ahmadi, A. Dehghan Nezhad,
	\textit{Some aspects of cosheaves on diffeological spaces},
	Categ. Gen. Algebr. Struct. Appl. 12(1) (2020) 123-147.
	
	\bibitem{ADD} A. Ahmadi, A. Dehghan Nezhad, B. Davvaz, \textit{Hypergroups and polygroups in diffeology}, Comm. Algebra 48(6) (2020) 2683-2698. 	
	
	\bibitem{BH}
	J.C. Baez, A.E. Hoffnung, 
	\textit{Convenient categories of smooth spaces},
	Trans. Amer. Math. Soc. 363(11) (2011) 5789-5825. 
	
	\bibitem{BT}
	R. Bott, L.W. Tu, \textit{Differential Forms in Algebraic Topology},
	Vol. 82. New York, Springer, 1982.
	
	
	
	\bibitem{CW}
	J.D. Christensen, E. Wu, 
	\textit{Smooth classifying spaces},
	Israel J. Math. 241(2)  (2021) 911-954.
	
	\bibitem{Coh}
	R.L. Cohen, \textit{The Topology of Fiber Bundles: Lecture Notes}, Stanford University, 1998, Available at
	\url{https://math.stanford.edu/~ralph/fiber.pdf}
	
	\bibitem{DA2018}
	A. Dehghan Nezhad, A. Ahmadi, 
	\textit{Cohomology theories and sheaves on diffeological spaces},
	49th Annual Iranian Mathematics Conference
	(2018).
	
	\bibitem{DA}
	A. Dehghan Nezhad, A. Ahmadi, 
	\textit{A novel approach to sheaves on diffeological spaces},
	Topology Appl. 263 (2019) 141-153.
	
	
	
	\bibitem{Hus}
	D. Husemoller, \textit{Fibre Bundles}, Graduate Texts in Mathematics, Vol. 20, Springer,
	New York, 1994
	
	
	\bibitem{PIZ2013}
	P. Iglesias-Zemmour, 
	\textit{Diffeology},
	Mathematical Surveys and Monographs  185 AMS, 2013.
	
	\bibitem{PIZ2022}
	P. Iglesias-Zemmour, \textit{\v{C}ech-De-Rham bicomplex in diffeology}, Israel J. Math. (2022) to appear. 
	
	\bibitem{KWW}
	D. Krepski, J. Watts, S. Wolbert, \textit{Sheaves, principal bundles, and \v{C}ech cohomology for diffeological spaces}, Preprint arXiv:2111.01032 (2022).
	
	\bibitem{M}
	S. Mac Lane, 
	\textit{Categories for the Working Mathematician},
	Springer-Verlag, 1998.
	
	
	
	\bibitem{Nic}
	L.I. Nicolaescu, \textit{Lectures on the Geometry of Manifolds}, World Scientific, 2008.
	
	
	\bibitem{JMS} 
	J.-M. Souriau, 
	\textit{Groupes diff\'{e}rentiels},
	in: Differential geometrical methods in mathematical physics, Proc. Conf., Aix-en-Provence/Salamanca, 1979, in:
	Lecture Notes in Math. 836, Springer Verlag, 1980, 91-128. 
	
	\bibitem{T}
	G. Tamme,  
	\textit{Introduction to \'{E}tale Cohomology},
	Translated by Manfred Kolster, 
	Springer-Verlag, 1994.
	
	
	
\end{thebibliography}
\end{document}